\documentclass[12pt]{iopart}

\usepackage{amsthm}
\usepackage{graphicx}
\usepackage{dsfont}
\usepackage{cases}
\usepackage{color}
\usepackage{stmaryrd}
\usepackage{textcomp}
\usepackage{moredefs}
\usepackage[lined,algonl,boxed]{algorithm2e}

\usepackage{todonotes}
\newcommand{\norm}[1]{\left\|#1\right\|}
\newcommand{\paren}[1]{\left(#1\right)}
\newcommand{\braces}[1]{\left\{#1\right\}}
\newcommand{\diffq}[2]{\frac{\partial #1}{\partial #2}}

\newcommand{\Mc}[1]{{\mathcal #1}}
\newcommand{\Eq}[1]{(\ref{eq:#1})}

\newcommand{\Fig}[1]{Figure~\ref{fig:#1}}
\newcommand{\Figs}[2]{Figures~\ref{fig:#1} and~\ref{fig:#2}}

\newcommand{\Sec}[1]{Section~\ref{sect:#1}}

\newcommand{\Lem}[1]{Lemma~\ref{lem:#1}}

\newcommand{\Cov}{\mathrm{Cov}}

\newcommand{\rv}[0]{\mathbf{r}}
\newcommand{\dv}[0]{\mathbf{d}}

\newcommand{\kv}[0]{\mathbf{k}}

\def\EE{\mathbb{E}}
\def\curl{\mathrm{\,curl}\,}
\def\curls{\mathrm{\,curl}}
\def\div{\mathrm{\,div}\,}
\def\divs{\mathrm{\,div}\!}
\def\grad{\mathrm{\,grad}\,}
\def\grads{\mathrm{\,grad}}

\newcommand{\curli}[1]{\,\mathrm{curl}_{\rho,#1}}
\newcommand{\divi}[1]{\,\mathrm{div}_{\rho,#1}}
\newcommand{\gradi}[1]{\,\mathrm{grad}_{\rho,#1}}
\newcommand{\argmin}{\mathrm{argmin}}
\newcommand{\diag}{\mathrm{diag}}
\newcommand{\blockdiag}{\mathrm{blockdiag}}
\newcommand{\trace}[1]{\mathrm{trace}\left(#1\right)}
\newcommand{\sol}{v}
\newcommand{\hatsol}{\widehat{\sol}}
\newcommand{\bsol}{\mathbf{\sol}}
\newcommand{\mom}{p}

\newcommand{\bmom}{\mathbf{\mom}}
\newcommand{\Xspace}{\mathbb{X}}
\newcommand{\Yspace}{\mathbb{Y}}
\newcommand{\lsp}{\left\langle}
\newcommand{\rsp}{\right\rangle}
\newcommand{\bg}{\mathbf{g}}
\newcommand{\bq}{\mathbf{q}}
\newcommand{\bw}{\mathbf{w}}
\newcommand{\qmom}{q}
\newcommand{\vl}{\sol_l}
\newcommand{\D}{\,\mathrm{d}}
\newcommand{\Mm}{\,\mathrm{Mm}}

\newtheorem{theorem}{Theorem}[section]
\newtheorem{lemma}[theorem]{Lemma}

\newtheorem{remark}[theorem]{Remark}
\newtheorem{definition}[theorem]{Definition}
\usepackage{iopams}  

\begin{document}

\title[Pinsker estimators for local helioseismology]%
{Pinsker estimators for local helioseismology: inversion of travel times for mass-conserving flows}

\author{Damien Fournier$^1$, Laurent Gizon$^{2,3,4}$,  Martin Holzke$^1$, Thorsten Hohage$^1$}

\address{$^1$ Institut f\"ur Numerische und Angewandte Mathematik, Georg-August-Universit\"at, \\ Lotzestr. 16-18, D-37083 G\"ottingen \\
$^2$ Max-Planck-Institut f\"ur Sonnensystemforschung, Justus-von-Liebig-Weg 3, 37077 G\"ottingen, Germany \\
$^3$ Institut f\"ur Astrophysik, Georg-August-Universit\"at G\"ottingen,  Friedrich-Hund-Platz 1, 37077 G\"ottingen, Germany \\
$^4$ National Astronomical Observatory of Japan, Mitaka, Tokyo 181-8588, Japan}
\ead{d.fournier@math.uni-goettingen.de}
\vspace{10pt}
\begin{indented}
\item[]\today
\end{indented}


\begin{abstract}
A major goal of helioseismology is the three-dimensional reconstruction of
the three velocity components of convective flows  in the solar interior from sets
of wave travel-time  measurements.
For small amplitude flows, the forward problem is described in good
approximation by a large system of convolution equations.
The input observations are highly noisy random vectors with a known dense
covariance matrix.
This leads to a large statistical linear inverse problem.

Whereas for deterministic linear inverse problems several computationally efficient minimax optimal 
regularization methods exist, only one minimax-optimal linear 
estimator exists for statistical linear inverse problems: the Pinsker estimator. However, it is often computationally inefficient because it requires 
a singular value decomposition of the forward operator or it is not applicable because of an unknown noise covariance matrix, 
so it is rarely used for real-world problems. 
These limitations do not apply in  helioseismology.
We present a simplified proof of the optimality properties of the Pinsker estimator and show that it yields significantly 
better reconstructions than traditional inversion methods used in helioseismology, i.e.\ Regularized Least Squares (Tikhonov regularization) 
and SOLA (approximate inverse) methods. 

Moreover, we discuss the incorporation of the mass conservation constraint in the Pinsker scheme 
using staggered grids. With this improvement we can reconstruct not only 
horizontal, but also vertical velocity components that are much smaller in amplitude. 
\end{abstract}

%
%
%
%
%

\section{Introduction}

Time-distance helioseismology aims at recovering the internal properties of the Sun from 
measurements of wave travel times between pairs of points \cite{DUV93}. The raw observations in helioseismology are time sequences of images of the line-of-sight velocity on the 
solar surface via Doppler shift measurements, for example from the Solar Dynamics Observatory (45 s cadence since 2010). 
These Doppler velocities contain information about the stochastic seismic wave field (acoustic waves and surface-gravity waves). 
Using a cross-correlation technique Duvall et al.~\cite{DUV93} showed that it is possible to measure the time it takes a wave 
packet to travel between any two points on the surface through 
the solar interior.  
The wave travel times are linked to 
(perturbations of) physical quantities 
via a large system of convolution equations. In this paper we focus on the estimation of flows. 
The inversion is traditionally performed using Tikhonov regularization \cite{TIK77} or the method 
of approximate inverse \cite{LOU90,SCH07} that are respectively called in the helioseismology community, Regularized Least Square 
(RLS) \cite{KOS96} and (Subtractive) Optimally Localized Averaging (OLA/SOLA) \cite{HAB04}.
The latter goes back to the Backus-Gilbert method \cite{BAC67} and, as pointed out by 
Chavent \cite{chavent:98}, it is also closely related to the method of sentinels introduced 
by J.L.~Lions for control problems (see \cite{lions:88a}). 

For overviews on linear statistical inverse problems we refer to \cite{CAV08,ES:02,tenorio:01}. 
Optimal rates of convergence for spectral regularization methods, in particular Tikhonov regularization, 
were shown in \cite{BIS07}, and for the CG method in \cite{BM:12}. 
Pinsker-type estimator for deconvolution problems on the real line were studied theoretically in 
different degrees of generality in a series of papers by Ermakov (see e.g.\ \cite{ERM90,ermakov:03}).  
The case of periodic deconvolution problems with noise in the operator was treated in \cite{CH:05}. A minimax estimator for spherical deconvolution over 
a reduced class of estimators was developed in \cite{HQ:15}.

For  linear inverse problems in Hilbert spaces with additive random noise Pinsker estimators 
are optimal in the following sense: For a given ellipsoid 
spanned by singular vectors of the forward operator, the Pinsker estimator minimizes 
the maximal risk (or expected square error) over this ellipsoid among all linear 
estimators. We point out that for deterministic inverse problems 
typically many optimal methods exist, e.g.\ Tikhonov regularization, some types of singular value 
decompositions, the Showalter methods and (asymptotically) Landweber iteration and Lardy's methods, 
each of course with an optimal choice of the regularization parameter or stopping index  
(see \cite{VAI87,TAU98}). In contrast, for statistical inverse problems,  the Pinsker method is the only minimax linear estimator (\cite{KO:71,PIN80}).
Moreover, it was shown by Pinsker \cite{PIN80} under mild assumptions that it is even asymptotically 
optimal among all (not necessarily linear) estimators if the noise is Gaussian.  
In most real world applications this estimator cannot 
be applied for two main reasons: First, it requires the computation of a Singular Value 
Decomposition (SVD) of the forward operator which is often not affordable due to the 
size of the problem. Second, the noise covariance matrix has to be known while only a poor 
estimate is generally available. This explains why other methods such as Tikhonov 
regularization or Conjugate Gradient methods are more often used for 
real world applications. However, these limitations are not problematic for the 
helioseismology problem studied here since the forward operator separates into a collection 
of small matrices for each spatial frequencies, for which an SVD can be computed in reasonable time,  
and the noise covariance matrix is known \cite{GIZ04, FOU14}. 
In this paper we will study the implementation and performance of 
Pinsker estimators for such problems. 

A notorious difficulty in local helioseismology is the inversion for vertical velocity components 
as their amplitude is much smaller than for horizontal velocities. The failure of inversion 
was reported in several publications using synthetic data (see e.g. \cite{ZHA07,DOM13}) and 
was explained by the crosstalk between the variables. 
Here, we show that incorporation of the mass conservation constraint in the 
Tikhonov or Pinsker methods allows to overcome these difficulties. We will discuss the 
implementation of mass conservation constraints with the help of staggered grids for 
the horizontal and the vertical velocity components.  

The plan of this paper is as follows: After introducing the physical background and the forward problem in 
\Sec{physics}, we describe in \Sec{helio} the inversion methods that are commonly used in this field so far.  
Then we introduce the Pinsker estimator in \Sec{pinsker} and present a simple proof that it is the unique 
minimax linear estimator. \Sec{mass} is devoted to the incorporation of the mass conservation constraint 
into this regularization scheme.
Finally, numerical results demonstrating the advantages of Pinsker methods compared to the state-of-the-art methods 
are discussed in \Sec{numerics}.

\section{Estimating flows by local helioseismology} \label{sect:physics}

In local helioseismology, it is acceptable to consider small patches of the solar surface and to neglect solar curvature. The domain of interest is approximated by a Cartesian box, with horizontal coordinates $\rv = (x,y)$ and vertical coordinate (height) $z$. Let us denote this domain by $V$.  Typically, $x$ and $y$ span several hundreds of megameters and $z$ several tens of megameters. 


The observables are time series of the line-of-sight velocities $\psi(\rv, t_j)$
at different points $\rv$ obtained from dopplergrams of the Sun's surface taken by satellites 
at equidistant time points $t_j$. 
From these quantities, we compute averaged travel times 
$\tilde{\tau}^a(\rv)$ at different points $\rv$ (and at time $t_0$, but we assume the time series $\psi$ to be stationary)
of the form 
\[
\tilde{\tau}^a(\rv) = \sum_j \int \Cov(\psi(\rv,t_0),\psi(\rv+\tilde{\rv},t_0+t_j))w^a(\tilde{\rv},t_j)
\D\tilde{\rv},\qquad a=1,\dots,N_a.
\]
(We reserve the symbol $\tau^a$ for differences of 
$\tilde{\tau}^a$ to a reference model.)
The weights $w^a$ are chosen such that $\tilde{\tau}^a(\rv)$ approximates a spatial average of the times a 
certain type of wave packet needs to travel from point $\rv$ to points $\rv+\tilde{\rv}$, see \cite{BOG97,DUV93} and the end of this section for more details. 
Hence, what will be called travel times in the following are linear functionals 
of the covariance operator of the observable $\psi$, written as $\tilde{\tau}^a = \Mc{W}_a(\mathbf{Cov}[\psi])$ or 
$\tilde{\tau} = \Mc{W}(\mathbf{Cov}[\psi])$ for the vector $\tilde{\tau} =(\tilde{\tau}^a)_{a=1..N_a}$ of all 
travel times. 

The observable $\psi$ is the image of the wave displacement $\bxi=\bxi(\rv,z, t)$ under 
an observation operator $\Mc{T}$, i.e.\ $\psi= \Mc{T}\bxi$. 
Ideally, $\psi(\rv,t)= l(\rv,t)\cdot \bxi(\rv,0,t)$ 
with the unit-length line-of-sight vector $l(\rv,t)$, but in practice $\Mc{T}$ also involves the point 
spread function of the instrument. 
%
%
The wave displacement is linked to internal properties of the Sun via a PDE describing the wave propagation in the Sun \cite{BUR15}:
\begin{equation*}
 \Mc{L}\bxi := \rho \left( \partial_t + \Gamma + \bsol \cdot \nabla \right)^2 \bxi - \nabla \left(  \rho c^2 \nabla \cdot \bxi \right) + \nabla \left( \bxi \cdot \nabla P \right) + \nabla \cdot \left( \rho g  \bxi \right) = f,
\end{equation*}
where $\rho$ is the density, $c$ the sound speed, $P$ the pressure, $g$ the gravitational acceleration, $\Gamma$ the damping, $\bsol$ the flow, and $f$ a (stochastic) source term responsible for the excitation of the seismic waves. Additional terms can be included to take into account the effects of rotation, magnetic field or a more complex form of the gravitational term.

Our aim is to recover the 3D flow velocity field $\bsol=(\sol^{\rm x},\sol^{\rm y},\sol^{\rm z})$ from observed travel times $\tilde{\tau}$. Inversion for other physical 
quantities can be performed analogously. We point out that actually computations are performed in the 
frequency domain, but at least formally we can write the forward operator as 
$F(\bsol) = \Mc{W}\big(\Mc{T}\Mc{L}[\bsol]^{-1}\mathbf{\Cov[f]}(\Mc{L}[\bsol]^{-1})^*\Mc{T}^*\big)$, 
so we have to solve the nonlinear operator equation $\tilde{\tau} = F(\bsol)+n$ where $n$ denotes noise. 
Under the assumption that $\bsol$ is small compared to the 
local wave speed, which is true at least in quiet parts of the Sun, the Born approximation $F(\bsol)\approx F(0)+ F'[0]\bsol$ 
is sufficiently accurate \cite{GIZ02}, and we obtain 
the linear operator equation $F'[0]\bsol = \tau+n$ with $\tau:=\tilde{\tau}-F(0)$. The operator 
$F'[0]$ can be written as an integral operator, and due to horizontal translation invariance 
the Schwartz kernel $K$ of $F'[0]$ only depends on the difference $\mathbf{r}-\mathbf{r}'$, so 
\begin{equation}
 \tau^a(\mathbf{r}) = \int_V \sum_{\beta\in\{\mathrm{x,y,z}\}} K^{a;\beta}(\mathbf{r}'-\mathbf{r}, z) \sol^\beta(\mathbf{r}',z) d^2\mathbf{r}' dz + n^a(\mathbf{r}), \qquad a=1,\ldots,N_a
 \label{eq:travelTime}
\end{equation}
(see \cite{GIZ02}). The functions $K^{a;\beta}$ are known as 
sensitivity kernels, but in contrast to the convention used in helioseismology where $K^{a;\beta}(\mathbf{r}'-\mathbf{r}, z)$ is replaced by $K^{a;\beta}(\mathbf{r}'+\mathbf{r}, z)$, we use a standard convolution integral as it is mathematically more convenient.
The assumption that the kernels are invariant under horizontal translation is intimately connected to the assumption that we are modeling only a small patch on the solar surface.

Due to mass conservation the flow velocity satisfies the equation 
\begin{equation}\label{eq:mass_conservation}
\div(\rho \bsol) = 0,
\end{equation}
where the mass density $\rho$ is assumed to depend on $z$ only. Note that this constraint reduces the 
effective number of unknowns of the inverse problem by about one third.

Besides the Born approximation we will use two further simplifying assumptions: 
The first approximation consists in imposing periodic boundary conditions in the horizontal 
variables. Since the kernels are localized, aliasing artifacts can be avoided by zero-padding, 
but, nevertheless, this approximation leads to a loss of information close to the boundaries. 
We may assume without further loss of generality that the periodicity cell is 
$[-\pi,\pi]^2$ in dimensionless coordinates. 
The second approximation consists in a discrete treatment of the depth variable $z$. 
For simplicity, we assume that the $\sol^{\beta}(\mathbf{r},\cdot)$ is represented by its values on a grid $\{z_1,\dots,z_{N_z}\}$ 
and define $\sol^{\beta,z_j}(\mathbf{r}):=\sol^{\beta}(\mathbf{r},z_j)$. 

Then, \Eq{travelTime} can be written  as
\begin{equation}
 \tau^a(\mathbf{r}) = \sum_{j=1}^{N_z} \sum_{\beta\in\{\mathrm{x,y,z}\}} \left(K^{a;\beta,z_j} \ast \sol^{\beta, z_j}\right)(\mathbf{r}) 
+ n^a(\mathbf{r}),
\qquad a=1,\ldots,N_a
 \label{eq:travelTime2}
\end{equation}
where $\ast$ denotes periodic convolution. Denoting by 
\begin{equation*}
 \sol_\kv:=(2\pi)^{-2}\int_{-\pi}^{\pi}\int_{-\pi}^{\pi} f(\mathbf{r})\exp(- i\mathbf{r}\cdot \kv)\,dx dy \: , \kv\in\mathbb{Z}^2,
\end{equation*}
the Fourier coefficients of a periodic function $f:(\mathbb{R}/\mathbb{Z})^2\to \mathbb{C}$, we can write \eref{eq:travelTime2} 
equivalently in Fourier space as 
\begin{equation}
 \tau^a_\kv = \sum_{j=1}^{N_z} \sum_{\beta\in\{\mathrm{x,y,z}\}} K^{a;\beta,z_j}_\kv \sol^{\beta, z_j}_\kv + n^a_\kv,
\qquad a=1,\ldots,N_a, \kv\in\mathbb{Z}^2. \label{eq:travelTimeFourier}
\end{equation}
The problem is now decoupled for each spatial frequency $\kv$ and can be written in a matrix form as
\begin{equation}
 \tau_\kv = K_\kv \bsol_\kv + n_\kv,
\qquad \kv\in\mathbb{Z}^2 \label{eq:pbIni}
\end{equation}
where the quantities we want to recover have been reorganized in the column vectors 
$\bsol_{\kv} = (\sol^{\beta,z_j}_\kv)_{\beta,z_j} \in 
\mathbb{C}^{3N_z}$,  the observables are $\tau_\kv = (\tau^a_\kv)_a \in \mathbb{C}^{N_a}$, and the Fourier transformed 
convolution kernels are $K_\kv = (K^{a;\beta, z_j}_\kv) \in \mathbb{C}^{N_a \times 3N_z}$. 

The noise is assumed to be translation invariant, so the noise covariance matrix, 
\[
\Lambda^{ab}(\dv) = \Cov[n^a(\rv), n^b(\rv+\dv)],\qquad a,b=1,\dots,N_a
\]
does not depend on $\rv$. As a consequence, noise vectors $n_\kv, n_{\kv'}$ for different  
spatial frequencies $\kv,\kv'\in \mathbb{Z}^2$ are uncorrelated, and the covariance matrix of $n_\kv$ is 
given by $\Lambda_\kv = (\Lambda^{ab}_\kv)_{ab} \in \mathbb{C}^{N_a\times N_a}$.  
An expression for these matrices was first derived in \cite{GIZ04} 
and generalized in \cite{FOU14} taking into account that the observation time is finite.


For our computations we will use the  kernels $K$ from \cite{SVA11}, which 
we are going to describe briefly.
We consider a Cartesian patch of the solar surface containing 200 $\times$ 200 
pixels with a spatial sampling width of $h_x = 1.46\Mm$. The vertical direction $z$ is discretized with $N_z = 89$ points 
using a variable step size as the variations are stronger close to the surface due to the density profile. This variation 
of the mass  density of several orders of magnitude near the surface is one of the 
difficulties to invert for velocities. The quantity $\mathbf{v} = (\sol^{\mathrm{x}}, \sol^\mathrm{y}, \sol^\mathrm{z})$ we want to recover has thus 
$3N_z=267$ degrees of freedom for each spatial frequency $\kv$.

In order to improve the signal-to-noise ratio, certain averages of point-to-point travel times are 
used, for example between 
the center of a disk and all the points located at a given radius of this disk. Such types of 
data are sensitive to in/out flows in this disk. Imposing other weights on the circle leads to 
data that are sensitive to East-West or North-South flows. 
Varying the center of this disk on the whole surface of the observational domain allow to build 
a map of observations. 
We use each of these three averaging schemes for $16$ radii from $5\Mm$ to $20\Mm$. Moreover, we use 
filters for f, p1, p2, p3, and p4 waves. (The first one is a gravitational wave, and the latters 
are acoustic waves with 1,2,3 or 4 nodes.) 
This yields $N_a=3\times 16\times 5=240$ travel time data for each of the $200^2$ points on the solar surface. 
Thus, the kernels $K_\kv$ are of the size $240 \times 267$ for each of the $200^2$ frequencies 
$\kv$. 

To provide some intuition for the problem we are solving, a representation of kernels for $v_x$ and $v_z$ using different filters is given in \Fig{kernel}. The right column represents cuts at $z=0$ (surface of the Sun) for the part sensitive to $v_x$ (top) and $v_z$ (middle). The kernels are localized around the center indicating that the data are relatively close to the quantities we want to infer for. The bottom plot shows the cross-talk betwween $v_z$ and $v_x$ i.e. how the data sensitive to $v_z$ are related to the ones for $v_x$. One can see that the amplitude is around ten times smaller than the one of the main kernel (top) and that the integral of this kernel is zero indicating that the average value of $v_z$ is not influenced by $v_x$. The left and right columns of \Fig{kernel} show the depth dependance of the kernels for different type of waves. One can see the importance of using different waves in order to probe different depths in the solar interior. However, all kernels are extremely sensitive to the surface making inversion at large depths highly ill-posed. 

\begin{figure}
\centering
 \includegraphics[ width=0.9\linewidth]{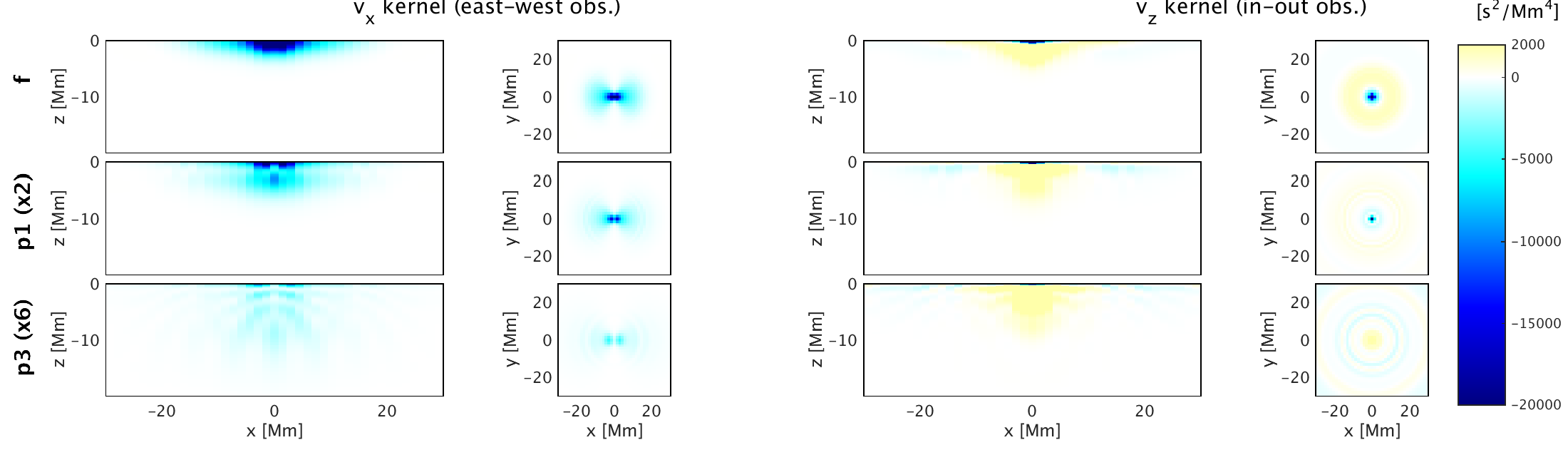}
 \caption{6 out of 240 sensitivity kernels describing the forward operator. 
Columns 1 and 3: cuts at $y=0$; columns 2 and 4: cuts at $z=0$. The different rows show different wave filters, from top to bottom: f, p1 and p3. To facilitate the comparison the amplitude of p1 (resp. p3) kernels have been multiplied by 2 (resp. 6). The largest the radial order the deepest is the sensitivity.
The columns 1 and 2 show East-West averaging schemes that are designed to be sensitive to 
$\sol^{\rm x}$ velocities, and columns 3 and 4 are in-out averaging sensitive  $\sol^{\rm z}$ velocities.}  \label{fig:kernel}
\end{figure}

An example of travel time map for a given filter is given in \Fig{travelTime} before adding the noise and after. 
The noise level corresponds to data averaged over 4 days with a temporal sampling of 45 seconds. Even with such a long averaging time, one can see that the noise is highly correlated, which underlines the importance of a good knowledge of the noise covariance matrix as computed in \cite{FOU14,GIZ04}.

 \begin{figure}
 \centering
 \includegraphics[trim=1cm 1.2cm 4cm 0cm, clip=true, width=0.9\linewidth]{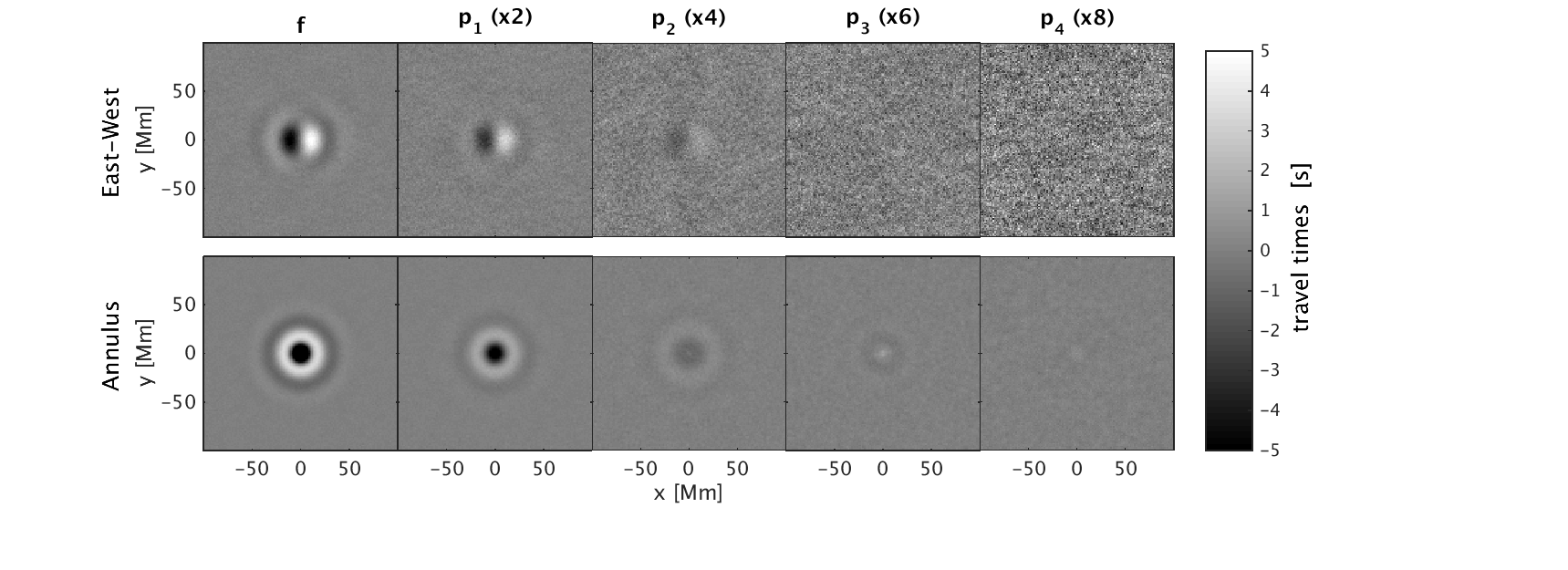} 
  \caption{Noisy synthetic travel times for a supergranule velocity model 
from Ref. \cite{DOM13}. The top label indicates the filter and the factor by which the data are multiplied (e.g. times 2 for p1 mode). The deeper the waves are travelling, the noisier the observations.} \label{fig:travelTime}
 \end{figure}


\section{Classical inversion methods used in local helioseismology} \label{sect:helio}

\subsection{Regularized Least Squares (RLS)}
Tikhonov regularization is generally called Regularized Least Squares (RLS) in the helioseismology community. 
Since the problem decouples for all $\kv$ (\cite{JAC99}), we can compute 
\begin{equation}\label{eq:RLS}
\eqalign{
\hatsol_\kv^{\rm RLS}&:= \argmin_{\sol_\kv}\left[
 \norm{\Lambda_\kv^{-\frac{1}{2}} (K_\kv \sol_\kv - \tau_\kv)}^2 + \alpha \norm{L_\kv \sol_\kv}^2\right] \\
&= \left(K_\kv^*\Lambda_\kv^{-1}K_\kv+\alpha L_{\kv}^*L_{\kv}\right)^{-1} K_\kv^*\Lambda_\kv^{-1}\tau_\kv
}
\end{equation}
independently for all spatial frequencies $\kv\in\mathbb{Z}^2$. 
Here $\alpha>0$ is the regularization parameter,  
and $L_\kv$ is a regularization matrix that can be the identity or the discretized version of the gradient or the
Laplacian in order to impose additional smoothness on the solution.

\subsection{Optimally Localized Averaging (OLA)}

Different types of Optimally Localized Averaging (OLA) methods are used in helioseismology. Recently, it was proposed to take 
advantage of the convolution in the horizontal space and to propose a multichannel OLA \cite{JAC12}. Similar to the 
previous approach the problem decouples for all frequencies and can be solved efficiently. We seek for a linear combination 
of travel times via weighting matrices $W_\kv = (W^{\beta, z_j;a}_\kv) \in \mathbb{C}^{3N_z \times N_a}$ 
(the Fourier coefficients of weighting kernels $W(\mathbf{r}):=\sum_{\kv\in\mathbb{Z}^2}W_\kv\exp(i \mathbf{r}\cdot \kv)$ with values in $\mathbb{R}^{3N_z \times N_a}$) such that
\begin{equation}\label{eq:SOLA}
 \hatsol_\kv = W_\kv \tau_\kv, \qquad \kv\in\mathbb{Z}^2
\end{equation}
is a good estimate of $\sol_\kv$. 

Note from the second line in \eref{eq:RLS} that RLS is also of this form 
with $W_{\kv}^{\mathrm{RLS}}=(K_\kv^*\Lambda_\kv^{-1}K_\kv+\alpha L_{\kv}^*L_{\kv})^{-1} K_\kv^*\Lambda_\kv^{-1}$. 
Inserting \Eq{travelTimeFourier} into \Eq{SOLA} yields 
\begin{equation}\label{eq:SOLA1}
 \hatsol_\kv =  W_\kv K_\kv \sol_\kv +  W_\kv n_\kv. 
\end{equation}

\begin{definition}\label{defi:averaging_kernel}
For a regularization method of the form \eref{eq:SOLA} 
the function $\Mc{K}(\mathbf{r}):=\sum_{\kv\in\mathbb{Z}^2}\Mc{K}_\kv\exp(i \mathbf{r}\cdot \kv)$ with 
Fourier coefficients 
 \begin{equation}
\Mc{K}_\kv:= W_\kv K_\kv
 \end{equation}
and values in $\mathbb{R}^{3N_z\times 3N_z}$ is called the \emph{averaging kernel} of the method. 
(Often only specific rows of $W$ and $\Mc{K}$ corresponding to a specific depth $z_j$ and a Cartesian 
component $\beta$ are considered. We will denote them by 
$W[\beta,z_j;:](\mathbf{r})$ and $\Mc{K}[\beta,z_j;:](\mathbf{r})$.) 
\end{definition}
Note from \eref{eq:SOLA1} that the expectation $\EE[\hatsol]$ and hence the bias $\EE[\hatsol]-\sol$ 
of the estimator $\hatsol$  is characterized by a convolution with the averaging kernel:
\[
\EE[\hatsol] = \mathcal{K}*\sol.
\]
To keep the bias small the diagonal entries ($\alpha=\beta$) of the averaging kernel $\Mc{K}^{\alpha,z_j,\beta,z_l}(\mathbf{r})$ 
should be well concentrated around $z_l \approx z_j$ and $\mathbf{r}=0$. The off-diagonal entries ($\beta\neq\alpha$)
measure the leakage from one Cartesian component $\beta$ to another component $\alpha$ and should be small. 

The SOLA (Subtractive OLA) methods aims at finding rows of a weighting kernel $W$ indexed by $\beta,z_j$
such that the corresponding rows of the averaging kernel 
$\Mc{K}$ are as close as possible to rows of a prescribed target function $\Mc{T}(\mathbf{r})\in \mathbb{R}^{3N_z\times 3N_z}$ 
while keeping the noise (last term in \Eq{SOLA1}) small. This can be achieved by setting 
\begin{eqnarray}\label{eq:SOLAmin}
W_\kv^{\rm OLA}[\beta,z_j;:] &:=& \argmin_{W\in\mathbb{C}^{1\times N_a}}\left[\|WK_\kv-\Mc{T}_\kv[\beta,z_j;:]\|^2
+ \mu W \Lambda_\kv W^*\right]
\end{eqnarray}
(see \cite{JAC12}) where $\mu>0$ is a trade-off parameter. Other objective functional can be chosen, see e.g.\ \cite{SVA11}. 
The target function $\Mc{T}^{\beta, z_j;\alpha,z_l}(\rv)$ for $\alpha=\beta$ is generally chosen as a Gaussian 
in $(\rv,z_l)$ around the point $(0,z_j)$. For $\alpha\neq\beta$ it is chosen as 0. Obviously, the convex quadratic minimization problem \eref{eq:SOLAmin} can be solved by solving the linear first order optimality conditions. 
We also mention the MOLA (Multiplicative OLA) \cite{CHR90} method which uses a product $\Mc{K} \Mc{T}$ instead of 
the difference.


These methods involve  the target functions $\Mc{T}$ and the parameter $\mu$ 
as parameters, the choices of which are not obvious and involve certain subjectivity. 
In the next 
section, we propose to use the Pinsker estimator that is optimal in the sense that it minimizes the risk in a 
given class of functions.

\section{Pinsker estimator} \label{sect:pinsker}

The problem described in Section \ref{sect:physics} can be formulated as a linear operator equation 
\begin{equation}
 \tau = K \sol + n. \label{eq:pbCont}
\end{equation}
in the Hilbert spaces $\mathbb{X} = L^2([-\pi,\pi]^2)^{3N_z}$ and $\mathbb{Y}= L^2([-\pi,\pi]^2)^{N_a}$ 
with a compact, linear operator $K:\mathbb{X}\to \mathbb{Y}$ given by a matrix of convolution operators. 

We assume that the noise  $n$ is a Hilbert space process in $\mathbb{Y}$ with zero mean value 
and known covariance operator $\Cov[n]$. The modelling errors that are ignored in the 
assumption $\mathbb{E}[n]=0$ and references for $\Cov[n]$ 
have been discussed in Section \ref{sect:physics}. 

An estimator is an operator $W:\mathbb{Y}\to \mathbb{X}$ that maps observations 
$\tau$ to an approximation $W\tau\in\mathbb{X}$ of $\sol$.  
The risk (or expected square error) of an estimator $W$ at $\sol$ is defined by
\begin{equation}
 R(W, \sol) = \mathbb{E} \left[ \norm{W(K\sol + n) - \sol}^2 \right]. \label{eq:risk}
\end{equation}
If $W$ is linear, the risk can be decomposed into a bias 
and a variance part using $\mathbb{E}[n] = 0$:
\begin{equation}
  R(W, \sol) = \norm{(WK - I)\sol}^2 + \mathbb{E}[\norm{Wn}^2]. \label{eq:biasVariance}
\end{equation}
The bias $\norm{(WK - I)\sol}^2$ describes how far $W$ is from the inverse of the forward operator 
while the variance term 
\(
\mathbb{E}[\norm{Wn}^2] = \trace{\Cov[Wn]} = \trace{W^*\Cov[n]W}
\)
describes the stochastic part of the error. 

The \emph{maximal risk} of an estimator $W$ on a set $\Theta \subset \mathbb{X}$ is defined as 
\begin{equation}
 R(W, \Theta) = \sup_{\sol \in \Theta} R(W, \sol) =  \sup_{\sol \in \Theta} \mathbb{E} \left[ \norm{W(K\sol + n) - \sol}^2 \right]. \label{eq:riskMax}
\end{equation}
The \emph{minimax risk} and the \emph{minimax linear risk} on $\Theta$ are obtained 
by taking the infimum over all estimators (or all linear estimator, resp.) of \Eq{riskMax}
\begin{equation}\label{eq:linearMinimaxRisk}
 R^N(\Theta) = \inf_W R(W, \Theta),\qquad 
 R^L(\Theta) = \inf_{W \: \rm{ linear}} R(W, \Theta). 
\end{equation}
A linear estimator $W$ that attains the infimum in \Eq{linearMinimaxRisk} 
is called a \emph{minimax linear estimator}. To construct such an estimator for \Eq{pbCont} 
we first perform a whitening by multiplying \Eq{pbCont} from the left by $\Cov[n]^{-1/2}$ to obtain 
\begin{equation}
 \tilde{\tau} = \tilde{K} \sol + \tilde{n} \label{eq:pbWhite}
\end{equation}
where $\tilde{\tau}:= \Cov[n]^{-1/2}\tau$ and
$\tilde{n}:= \Cov[n]^{-1/2}n$ is now a white noise process, i.e.\ $\Cov[\tilde{n}] = I_{\mathbb{Y}}$. 
To ensure that $\tilde{K}:=\Cov[n]^{-1/2}K$ is well defined, we assume that $\Cov[n]$ is strictly positive 
definite, i.e.\ every linear functional of $\tau$ contains a minimal fixed amount of noise. 
Although this assumption could be relaxed, it is simple and intuitive, and also guarantees compactness 
of $\tilde{K}$. Hence, $\tilde{K}$ admits a 
singular value decomposition $\left\{ (\sigma_l, \varphi_l, \psi_l)\, :\, l \in \mathbb{N} \right\}$. 
This allows us to rewrite the operator equation \Eq{pbCont} as a diagonal operator equation 
in sequence spaces given by 
\begin{equation}
 y_l = \sigma_l \vl +  n_l \label{eq:pbPinsker}
\end{equation}
with observables $y_l:=\langle \Cov[n]^{-1/2}\tau,\psi_l\rangle_{\mathbb{Y}}$ and unknowns 
$\vl:=\langle \sol,\varphi_l\rangle_{\mathbb{X}}$. Due to Gaussianity 
the noise $(n_l)_{l \in \mathbb{N}}$ is a sequence of uncorrelated  $\Mc{N}(0,1)$ random variables. 
Let us consider linear diagonal estimators of the form 
\begin{equation}\label{eq:diag_estimator}
 \hat{\sol}_l := \frac{\lambda_l}{\sigma_l} y_l, \qquad 
W_{\lambda}\tau := \sum_{l\in\mathbb{N}} \frac{\lambda_l}{\sigma_l}\langle\Cov[n]^{-1/2}\tau,\psi_l \rangle_{\mathbb{Y}} \varphi_l
\end{equation}
with weights $\lambda_l\in\mathbb{R}$. 
The risk  $R(\lambda, \sol) :=  R(W_{\lambda}, \sol)$ of such estimators is given by 
\begin{equation}
 R(\lambda, \sol) = 
 \sum_{l\in \mathbb{N}} \left[ (1 - \lambda_l)^2 \vl^2 + \frac{\lambda_l^2}{\sigma_l^2}  \right].  \label{eq:biasVariance2}
\end{equation}
We will consider ellipsoids of the form 
 \begin{equation}\label{eq:ellipsoid}
  \Theta := \left\{ \sol\in\mathbb{X}: \sum_{l=1}^\infty a_l^2 \vl^2 \leq Q \right\}
 \end{equation}
with $Q>0$ and $a_l>0$. Then the risk  $R(\lambda, \Theta) :=  R(W_{\lambda}, \Theta)$ 
is given by 
\begin{equation}\label{eq:risk_pinsker}
R(\lambda,\Theta) = Q\sup_{l\in\mathbb{N}}\frac{(1-\lambda_l)^2}{a_l^2} 
+ \sum_{l=1}^{\infty}\frac{\lambda_l^2}{\sigma_l^2}. 
\end{equation}
\begin{lemma}\label{lem:diagonal_form}
Any minimax linear estimator must be of the diagonal form \Eq{diag_estimator}.
\end{lemma}
\begin{proof} 
Note that since $a_l\to\infty$, the supremum in \Eq{risk_pinsker} is attained at some index 
$l_{\lambda}\in\mathbb{N}$, and $R(\lambda,\Theta) = R(\lambda,\{v_{\lambda}\})$ with 
$v_{\lambda}=(\sqrt{Q}/a_{l_{\lambda}})\varphi_{l_{\lambda}}\in\Theta$. 
If a linear estimator $W$ with a nondiagonal 
(infinite) matrix representation is replaced by its diagonal part $\diag(W)$, 
the bias part of $R(W,\{v_{\diag(W)}\})$ cannot increase and the variance part 
strictly decreases. Hence, 
\[
R(\diag(W),\Theta)=R(\diag(W),\{v_{\diag(W)}\})<R(W,\{v_{\diag(W)}\})
\leq R(W,\Theta),
\]
which shows that $W$ is not minimax. 
\end{proof}

Even though the following result is well-known, we would like to present a short proof 
since we consider it more instructive and simpler than other proofs, e.g.\ in 
\cite{BL:95,PIN80,TSY09} (all for the equivalent regression problems version of the theorem). 
In particular, we derive the formulas \Eq{pinsker_weights} and \Eq{pinskerCte} and not 
just verify them, and it becomes apparent that 
$\overline{\kappa}\sqrt{Q}$ is the bound on the bias.  

\begin{theorem}[Pinsker estimator] \label{th:pinsker}
Consider a sequence $(a_l)_{l\in\mathbb{N}}$ such that $a_l > 0$ and 
$\lim_{l \rightarrow \infty} a_l = \infty$, and an ellipsoid of the form \Eq{ellipsoid}
with $Q>0$. 
Then there exists a unique minimax linear estimator on $\Theta$.  
It is of the form \eref{eq:diag_estimator}, and its weights are given by 
\begin{equation}\label{eq:pinsker_weights}
 \overline{\lambda}_l =\max( 1 - \overline{\kappa} a_l, 0), 
\end{equation}
 where the constant $\overline{\kappa} > 0$ is the unique solution of the equation
\begin{equation}
 \kappa Q - \sum_{l=1}^\infty \frac{a_l}{\sigma_l^2} \max(1 - \kappa a_l, 0) = 0. \label{eq:pinskerCte}
\end{equation}
The minimax linear risk is given by
\( R^L(\Theta) = \sum_{l=1}^\infty \frac{1}{\sigma_l^2} \max(1 - \overline{\kappa} a_l, 0).
\)
\end{theorem}

\begin{proof}
The infimum of $R(\lambda,\Theta)$ over all sequences $\lambda$ can be reduced to the set 
$\Lambda:= \{\lambda\in l^2(\mathbb{N}): \|\lambda\|_{\infty}\leq 1\}$ since 
$R(\lambda,\Theta)=\infty$ if $\lambda\notin l^2(\mathbb{N})$ 
and $R(\lambda,\Theta)$ strictly decreases 
if some $\lambda_j\notin [-1,1]$ is replaced by its metric projection onto $[-1,1]$. 
Let us introduce the decomposition 
\[
\Lambda = \bigcup_{0<\kappa\leq 1/\underline{a}} \Lambda_\kappa 
\qquad\mbox{with}\qquad 
\Lambda_\kappa:= \left\{\lambda\in\Lambda: \sup_{j\in\mathbb{N}}\frac{|1-\lambda_j|}{a_j}=\kappa
\right\}
\]
with $\underline{a}:=\min_{j}a_j$. 
 In view of \Eq{risk_pinsker} we have 
$R(\lambda,\Theta) = \kappa^2Q+ \sum_{l=1}^\infty(\lambda_l/\sigma_l)^2$ 
for $\lambda\in \Lambda_\kappa$, so the infimum over
$\lambda\in \Lambda_\kappa$ is attained if and only if 
$\lambda_l= \argmin_{|1-x|\leq\kappa a_l}x^2=\max(1-\kappa a_l,0)$ for all $l\in\mathbb{N}$. Note that this is \Eq{pinsker_weights} if $\kappa=\overline{\kappa}$. Using this formula for 
the minimizer we find that
\[
\inf_{\lambda\in\Lambda_{\kappa}}R(\lambda,\Theta)= \varphi(\kappa)
\qquad \mbox{with}\qquad 
\varphi(\kappa):=\kappa^2Q+\sum_{l=1}^{\infty}\frac{\max(1-\kappa a_l,0)^2}{\sigma_l^2}\,.
\]
Therefore $\inf_{\lambda\in\Lambda} R(\lambda,\Theta)
= \inf_{0< \kappa\leq1/\underline{a}}\varphi(\kappa)$.  
Note that $\varphi$ is strictly convex and differentiable with 
$\varphi'(\kappa)$ given by the left hand side of 
\Eq{pinskerCte} since the sum is finite in a neighborhood of 
any $\kappa>0$. 
Moreover, $\lim_{\kappa\searrow 0}\varphi(\kappa)=\infty$ 
and $\varphi'(1/\underline{a})= Q/\underline{a}>0$. Therefore,  $\varphi$ attains 
its infimum on $(0,1/\underline{a}]$ at the unique solution $\bar{\kappa}$ 
to $\varphi'(\kappa)=0$. 
\end{proof}

Instead of the implicit equation \Eq{pinskerCte} for $\overline{\kappa}$ there is also 
the following explicit formula if the sequence $(a_l)$ is non-decreasing 
(see \cite{TSY09}): 
\[
\fl
  \overline{\kappa} = \frac{\sum_{j=1}^N \frac{a_j}{\sigma_j^2}}{Q + \sum_{j=1}^N \frac{a_j^2}{\sigma_j^2}}
\qquad \mbox{with}\qquad 
N := \max\left\{n\in \mathbb{N}:
\sum_{l=1}^n \frac{1}{\sigma_l^2}a_l (a_n - a_l)<Q\right\}.
\]
From a practical point of view, this formula is only useful if $Q$ is known exactly. 
But this is a rather unrealistic assumption.  $Q$ should rather be seen as a regularization 
parameter. But since there is a one-to-one correspondence between $Q$ and $\overline{\kappa}$ 
via \Eq{pinskerCte}, it is much simpler to consider $\overline{\kappa}$ as regularization parameter. The choice of regularization parameters is an important and well-studied problem, 
but since the focus of this paper is on the comparison of regularization methods, we do not further discuss it here. 

A comparison of the linear minimax risk $R^L$ with the nonlinear one was given in \cite{PIN80}. Under the additional assumptions that the noise is Gaussian and that
\begin{equation}
 \sup_{j \in \mathbb{N}} \frac{\sum_{j=1}^J \sigma_j^2}{\sup_{j \leq J} \sigma_j^2} < \infty, \label{eq:hypPinsker}
\end{equation}
then $R^L(\Theta) \sim R^N(\Theta)$ as the noise level tends to 0. 
Assumption \Eq{hypPinsker} was later relaxed to $\sup_j \sigma_j / \sigma_{j+1} < \infty$ 
\cite{GOL99}. This assumption is very plausible in the context of our problem. 

It remains to discuss the choice of the ellipsoid $\Theta$. Without depth inversion, i.e.\ for $N_z=1$ and a scalar physical 
quantity, it is natural to define $\Theta$ in terms of some bound on the power spectrum of the form 
\[
\sum_{\kv\in\mathbb{Z}^2} \gamma(\kv) |\bsol_\kv|^2\leq Q.
\]
E.g.\ for the choice $\gamma(\kv) = (1+|\kv|^2)^s$ the ellipsoids $\Theta$ are balls in the periodic Sobolev $H^s([-\pi,\pi]^2)$. 
In depth direction admissible choices of $\Theta$  are more difficult to interpret 
since the axes of the ellipsoid must coincide with the singular vectors of the 
forward operators. 

\begin{figure}%
\hspace*{-6ex}
\includegraphics[width=1.2\columnwidth]{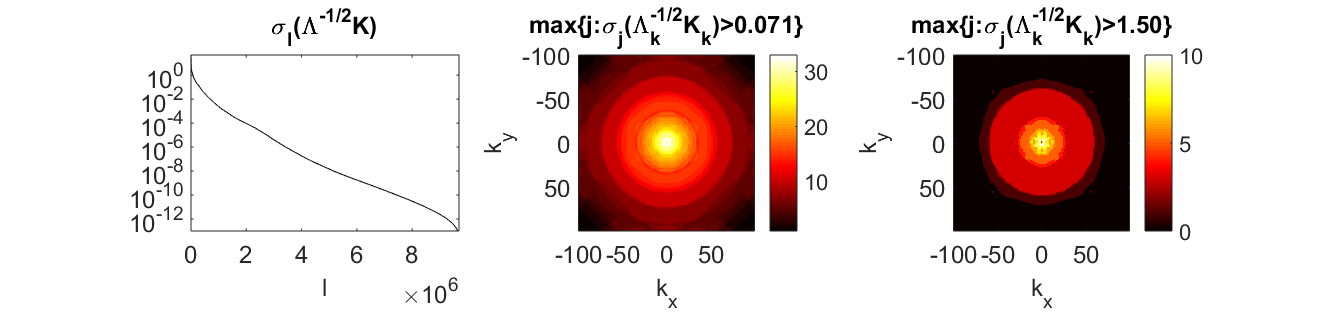}%
\caption{Left panel: Singular values of the forward operator with whitening. 
Note the exponential decay of the singular values. 
Middle and right panel: Distribution of singular values in frequency space. 
The color at each spatial frequency $\kv=(k_x,k_y)$ represents the number of 
singular values $> 0.071$ and $>1.5$ (with $\sigma_1=41.6$). 
These thresholds correspond to positive Pinsker weights and 
Pinsker weights $\overline{\lambda}_l\geq 0.5$, respectively, in 
\Figs{vzReconstructedNoDiv}{vxReconstructed}.}%
\label{fig:sv}%
\end{figure}

We choose the weights $a_l$ such that $a_l^2$ grows asymptotically as the 
weights $\gamma(\kv)=\gamma(\kv(l))$ on the $L^2$-Fourier coefficients 
defining an $H^1(V)$-ball in a cuboid $V\subset\mathbb{R}^3$ 
as $l\to \infty$, i.e.,
\begin{equation}\label{eq:choice_ellipsoid}
a_l= l^{1/3}. 
\end{equation}
Here $\kv(l)$ denotes an enumeration of 
the three-dimensional spatial frequencies such that $|\kv(l)|$ is non-decreasing. 
Empirically, we observe that the singular 
values $\sigma_l$ of our forward operator decay exponentially, i.e.\ 
$\sigma_l\sim \exp(-\alpha-\beta l)$ for some $\alpha\in\mathbb{R}$ and $\beta>0$,  
and their ordering at least roughly corresponds to the ordering described 
by $\kv(l)$ (see \Fig{sv}). 

\section{Mass conservation constraint} \label{sect:mass}

In this section we discuss how the mass conservation constraint 
$\div(\rho \bsol)=0$ mentioned in \eref{eq:mass_conservation} 
can be incorporated into Tikhonov regularization and the Pinsker method. 
We will assume that $0<\rho_{\min}\leq \rho\leq\rho_{\max}<\infty$,  
that $\rho$ is smooth and depends only on $z$. Then the inverse problem can be formulated as 
\begin{equation}\label{eq:IP_mass_conservation}
K \mathbf{v} = \tau \quad \mbox{subject to }\div(\rho \mathbf{v}) = 0.
\end{equation}

In an abstract Hilbert space setting an equality constraint $B\bsol=0$ with a bounded 
linear operator $B:\mathbb{X}\to\mathbb{Z}$ does not change much since we can simply replace
$\mathbb{X}$ by the null-space $\Mc{N}(B)$ of $B$. However, as it is often inconvenient 
to explicitly construct a basis of $\Mc{N}(B)$, it is preferable to work in the larger 
space $\mathbb{X}$. 

E.g.\ statistical Tikhonov regularization with noise covariance operator $\Lambda$ and 
a (differential) operator $L$ mapping to a Hilbert space 
$\mathbb{V}$ applied to \eref{eq:IP_mass_conservation} reads
\[
\bsol_{\alpha}=\argmin_{B\bsol = 0}\left[\|\Lambda^{-1/2}K\bsol-\tau\|_{\mathbb{Y}}^2
+\alpha \|L\bsol\|_{\mathbb{V}}^2\right].
\]
To treat the side condition we consider the corresponding Lagrange function 
$\mathcal{\mathcal{L}}(\bsol,\mu):= \|\Lambda^{-1/2}K\bsol-\tau\|_{\mathbb{Y}}^2+
\alpha \|L\bsol\|_{\mathbb{V}}^2+\langle\mu,\alpha B\bsol\rangle_{\mathbb{Z}}$ with a 
Lagrange multiplier $\mu\in \mathbb{Z}$. Here $B$ has been multiplied by the regularization 
parameter $\alpha$ to improve the condition number of the 
optimality conditions $\diffq{\mathcal{L}}{\bsol}=0$ and 
$\diffq{\mathcal{L}}{\mu}=0$. These then  lead to the saddle point equation
\[
\left(\begin{array}{cc}
K^*\Lambda^{-1}K+\alpha L^*L & \alpha B^* \\
\alpha B & 0
\end{array}\right)
\left(\begin{array}{c}
\bsol_{\alpha} \\ \mu
\end{array}\right)
= \left(\begin{array}{c}
K^*\Lambda^{-1}\tau \\ 0
\end{array}\right).
\]

\subsection{Fully continuous setting}
In this subsection we discuss a continuous treatment of the depth variable $z$. 
If $V=[-\pi,\pi]^2\times[z_{N_z},z_{0}]$ is the domain of interest, we may choose 
\begin{eqnarray*}
\mathbb{X}&=&\left\{\bsol\in H^1(V)^3:\bsol(\cdot;z) \mbox{  periodic, }
\sol^{\mathrm{z}}(\cdot,z_0)=\sol^{\mathrm{z}}(\cdot,z_{N_z})=0
\right\}.
\end{eqnarray*}
This choice of boundary conditions rules out coronal mass ejections, which are very 
simple to detect and for which the 
Born approximation used in the derivation of the forward operator breaks down anyways. 

We equip $\Xspace$ with the norm 
$\|\bsol\|_{\mathbb{X}}:=\langle\rho\bsol,\rho\bsol\rangle_{H^1}^{1/2}$ where
\[
\langle \rho \bsol,\rho \bw \rangle_{H^1}:=
\left(\sum_{\beta\in\{\rm x,y,z\}}\langle\rho \sol^{\beta}, \rho w^{\beta}\rangle_{L^2(V)}+
\langle\grad \rho \sol^{\beta},\grad \rho w^{\beta}\rangle_{L^2(V)^3}\right)^{1/2}.
\]
Under our assumptions on $\rho$ the norms $\|\rho\bsol\|_{H^1}$ and $\|\bsol\|_{H^1}$ are 
equivalent, but since $\rho$ varies over several orders of magnitude, the incorporation 
of $\rho$ in the norm makes a significant difference. 

Let us introduce the operators $\grads_{\rho}\bsol:= \grad(\rho\bsol)$, 
$\curls_{\rho}\bsol:= \curl(\rho\bsol)$, and $\divs_{\rho}\bsol:= \div(\rho\bsol)$. 
The following lemma  summarizes the properties of the subspace 
$\Mc{N}(\divs_{\rho})\subset \mathbb{X}$ and will be proved in an appendix. 
\begin{lemma}\label{lem:helmholtz_cont}
\begin{enumerate}
\item\label{it:grad_curl_div}
For all $\bsol,\bw\in\mathbb{X}$ we have
\begin{equation}\label{eq:grad_curl_div}
\fl
\sum_{\beta\in\{\rm x,y,z\}}
\langle\grads_{\rho} \bsol^{\beta},\grads_{\rho} \bw^{\beta}\rangle_{L^2(V)^3}
= \langle \curls_{\rho}\bsol,\curls_{\rho} \bw\rangle_{L^2(V)^3}  
+ \langle \divs_{\rho} \bsol,\divs_{\rho} \bw\rangle_{L^2(V)}
\end{equation}
\item\label{it:norm_equivalence}
There exists a constant $c>0$ such that the inequalities 
\begin{equation}\label{eq:norm_equivalence}
\fl
c\|\rho\bsol\|_{H^1}^2\leq\|\curls_{\rho}\bsol\|_{L^2}^2 
+\frac{1}{|V|}\sum_{\beta\in\{{\rm x,y}\}}\left(\int_V \rho\sol^{\beta}\mathrm{d}(\mathbf{r},z)\right)^2
\leq \|\rho\bsol\|_{H^1}^2
\end{equation}
hold true for all $\bsol\in\mathbb{X}$ with $\divs_{\rho} \bsol =0$. 
\item\label{it:ortho_cont} 
$\mathbb{X}_{\diamond}:=\{\bsol\in\mathbb{X}:\int_V\rho\sol^{\rm x}d(\mathbf{r},x) = 
\int_V\rho\sol^{\rm y}d(\mathbf{r},x)=0\}$ has the Helmholtz decomposition
\[
\mathbb{X}_{\diamond} = \mathcal{N}_{\diamond}(\divs_{\rho})\oplus\mathcal{N}_{\diamond}(\curls_{\rho})
\]
with $\mathcal{N}_{\diamond}(\divs_{\rho}):=\{\bsol\in\mathbb{X}_{\diamond}:\divs_{\rho}\bsol=0\}$ 
and $\mathcal{N}_{\diamond}(\curls_{\rho}):=\{\bsol\in\mathbb{X}_{\diamond}:\curls_{\rho}\bsol=0\}$. These subspaces 
are orthogonal both with respect to the $\Xspace$ inner product and the inner product 
$\lsp \rho \bsol,\rho \bw\rsp_{L^2(V)}$. 
\end{enumerate}
\end{lemma}

We will choose $L\bsol:=\curl(\rho\bsol)$. This means we do \emph{not} incorporate the means of 
the horizontal velicity components into the penalty term, which are needed to obtain a norm on 
$\{\bsol\in\mathbb{X}:\divs_\rho\bsol=0\}$. 
This is justified as the data are sensitive to constant horizontal velocities, i.e.\ 
$K$ restricted to ${\rm span}\{(1/\rho,0,0),(0,1/\rho,0)\}$ is bounded from below 
(see \cite[\S 8.2]{EHN:96}).

\subsection{(Semi-) discrete approximation}
In this subsection we discuss a discrete approximation of the $z$-variable which inherits 
the essential properties of the continuous setting. 
We found this crucial for good numerical results. 
Since $\rho$ depends only on $z$, the constraint $\div(\rho\bsol)=0$ separates into 
\[
ik_{\mathrm{x}} \rho \sol^{\mathrm{x}}_\kv + ik_{\mathrm{y}} \rho \sol^{\mathrm{y}}_\kv 
+ \diffq{\rho \sol^{\mathrm{z}}_\kv}{z} = 0,
\qquad \kv=(k_{\mathrm{x}},k_{\mathrm{y}})\in\mathbb{Z}^2. 
\]
Hence the only difference between a continuous and a discrete treatment of the 
(periodic) horizontal variables $x$ and $y$ is that in the former case infinitely many 
spatial frequencies must be considered, and in the latter case only finitely many. 


It will be essential to use different grids for the horizontal and the vertical velocities 
to preserve the most important properties of the continuous setting as summarized in 
\Lem{helmholtz_cont} in the discrete setting. 
For a given grid $z_0>z_1>\dots>z_{N_z}$ in vertical direction we introduce the midpoints 
$z_{j+1/2}:= \frac{1}{2}(z_j+z_{j+1})$. The horizontal velocities will be represented on 
$\{z_{1/2},\dots, z_{N_z-1/2} \}$ whereas the vertical velocities will be represented by their values on 
$\{z_1,\dots,z_{n-1}\}$. Here the points $z_0$ and $z_{N_z}$ have been omitted due to the 
Dirichlet boundary conditions for $\sol^{\rm z}$ such that 
\[
\sol^{\rm x}_\kv,\sol^{\rm y}_\kv\in \mathbb{V}:=\mathbb{C}^{N_z},\qquad 
\sol^{\rm z}_\kv \in \mathbb{W}:=\mathbb{C}^{N_z-1}. 
\]
These quantities will be indexed by $\sol^{\rm z}_{\kv} = (\sol^{\mathrm{z}}_{\kv,1},\dots,
\sol^{\mathrm{z}}_{\kv,N_z-1})^{\top}$ and 
$\sol^{\beta}_{\kv} = (\sol^{\beta}_{\kv,1/2},\dots,\sol^{\beta}_{\kv,N_z-1/2})^{\top}$, $\beta = x,y$. 
To define inner products on $\mathbb{V}$ and $\mathbb{W}$ 
we introduce weights $\delta_j:=z_{j-1/2}-z_{j+1/2}$ for $j=1,\dots, N_z-1$ and 
$\delta_{j+1/2}:=z_{j}-z_{j+1}$ for $j=0,\dots,N_z-1$. Then we introduce Gram matrices 
\begin{equation}\label{eq:GV_GW}
G_{\mathbb{V}}:=\diag(\delta_{1/2},\dots,\delta_{N_z-1/2})\quad \mbox{and}\quad
G_{\mathbb{W}}:=\diag(\delta_1,\dots,\delta_{N_z-1})
\end{equation}
 defining inner products 
$\langle v_1,v_2\rangle_{\mathbb{V}}:=v_2^{*}G_\mathbb{V}v_1$ on $\mathbb{V}$ and 
$\langle v_1,v_2\rangle_{\mathbb{W}}:=v_2^{*}G_\mathbb{W}v_1$ on $\mathbb{W}$. 
Similarly, we define $\rho_{j}:=\rho(z_j)$ for $j\in\{0,1/2,1,\dots,N_z\}$ 
and the matrices $M_{\rho}^\mathbb{V} = \diag(\rho_{1/2}, \ldots \rho_{N_z-1/2})$ and $M_{\rho}^\mathbb{W} = \diag(\rho_{1}, \ldots \rho_{N_z-1})$.
We approximate derivatives by the finite differences 
\[
\fl
\diffq{ \sol^{z}_\kv}{z}(z_{j+1/2}) \approx 
\frac{\sol^{{\rm z}}_{\kv,j}- \sol^{{\rm z}}_{\kv,j+1}}{\delta_{j+1/2}} 
= (D_{z}^{\mathbb{W}}\sol^{\mathrm{z}}_\kv)_j,\qquad 
\diffq{\sol^{\beta}_\kv}{z}(z_j) \approx
\frac{\sol^{\beta}_{\kv,j-1/2}- \sol^{\beta}_{\kv,j+1/2}}{\delta_j}
=(D_{z}^{\mathbb{V}}\sol^{\beta}_\kv)_j
\]
for $\beta={\rm x},{\rm y}$  
with $D_{z}^{\mathbb{W}}\in \mathbb{C}^{N_z\times (N_z-1)}\doteq L(\mathbb{W},\mathbb{V})$ 
and $D_{z}^{\mathbb{V}}\in \mathbb{C}^{(N_z-1)\times N_z}\doteq L(\mathbb{V},\mathbb{W})$ given by 
\begin{eqnarray*}
\fl 
D_{z}^{\mathbb{V}} &:=& G_{\mathbb{W}}^{-1}\left(\begin{array}{cccc}
1 & -1 \\
& \ddots &\ddots \\
& & 1 &-1
\end{array}\right)\quad \mbox{and}\quad 
D_{z}^{\mathbb{W}} := G_{\mathbb{V}}^{-1}
\left(\begin{array}{cccc}
-1 \\
1 & -1 \\
& \ddots &\ddots \\
& & 1 & -1 \\
& &   &1
\end{array}\right).
\end{eqnarray*}
These matrices are skew-adjoint with respect to the inner products in $\mathbb{V}$ and $\mathbb{W}$ since 
$G_{\mathbb{V}}D_{z}^{\mathbb{W}} = -(G_{\mathbb{W}}D_{z}^{\mathbb{V}})^*$, 
and hence
\begin{equation}\label{eq:adjointnessD}\fl
\langle D_{z}^{\mathbb{W}}w,v\rangle_{\mathbb{V}} 
= v^*G_{\mathbb{V}}D_{z}^{\mathbb{W}}w 
= -v^*(G_{\mathbb{W}}D_{z}^{\mathbb{V}})^*w
= -(D_{z}^{\mathbb{V}}v)^*G_{\mathbb{W}}w = -\langle w, D_{z}^{\mathbb{V}}v\rangle_{\mathbb{W}}. 
\end{equation}
Now we introduce the following approximations to the $\div$, $\grad$, $\curl$, and $\curl^*$  
for the spatial frequency $\kv\in\mathbb{Z}^2$:
\begin{eqnarray*}
&\divs_{\kv}\;:= \left(\begin{array}{ccc}
ik_{\rm x}I^{\mathbb{V}}\;\;\;\; & ik_{\rm y} I^{\mathbb{V}}\;\;\;\;&  D_{z}^{\mathbb{W}} \end{array}\right): 
\mathbb{V}\times \mathbb{V}\times \mathbb{W}\to \mathbb{V},\\
&\grads_{\kv}\;:=\left(\begin{array}{ccc}
ik_{\rm x}I^{\mathbb{V}}\;\;\;\; & ik_{\rm y} I^{\mathbb{V}}\;\;\;\;&  (D_{z}^{\mathbb{V}})^{\top} \end{array}\right)^{\top}: 
\mathbb{V}\to \mathbb{V}\times \mathbb{V}\times \mathbb{W}\\
&\curls_{\kv}:= \left( \begin{array}{ccc}
                   0 & -D_{z}^{\mathbb{V}} & ik_{\rm y} I^{\mathbb{W}} \\
                   D_{z}^{\mathbb{V}} & 0 & -ik_{\rm x} I^{\mathbb{W}} \\
                   -ik_{\rm y} I^{\mathbb{V}} & ik_{\rm x} I^{\mathbb{V}} & 0
                  \end{array} \right):
									\mathbb{V}\times \mathbb{V}\times \mathbb{W}\to \mathbb{W}\times \mathbb{W}\times\mathbb{V},\\
&\curls_{\kv}^{\#} := \left( \begin{array}{ccc}
                   0 & -D_{z}^{\mathbb{W}} & ik_{\rm y} I^{\mathbb{V}} \\
                   D_{z}^{\mathbb{W}} & 0 & -ik_{\rm x} I^{\mathbb{V}} \\
                   -ik_{\rm y} I^{\mathbb{W}} & ik_{\rm x} I^{\mathbb{W}} & 0
                  \end{array} \right):
				\mathbb{W}\times \mathbb{W}\times \mathbb{V}\to \mathbb{V}\times \mathbb{V}\times\mathbb{W}.
\end{eqnarray*}
Let us introduce the spaces 
$\mathbb{X}_{\kv}:= \mathbb{V}\times \mathbb{V}\times \mathbb{W}$ and 
$\mathbb{Y}_{\kv}:=\mathbb{W}\times \mathbb{W}\times\mathbb{V}$,  
the multiplication operator 
$M_{\rho}^{\Xspace}:= \blockdiag\Big(M_{\rho}^{\mathbb{V}},\,M_{\rho}^{\mathbb{V}},\,
M_{\rho}^{\mathbb{W}}\Big)$, and the mappings 
\begin{eqnarray}
\label{eq:divrho_curlrho}
&&\divi{\kv}:= \divs_{\kv} M_{\rho}^{\Xspace},\qquad\qquad 
\curli{\kv}:=  \curls_{\kv}M_{\rho}^{\Xspace},\\ 
&&\curli{\kv}^{\#}:= (M_{\rho}^{\Xspace})^{-1}\curls_{\kv}^{\#},\qquad\qquad
\gradi{\kv}:= (M_{\rho}^{\Xspace})^{-1}\grads_{\kv}.\nonumber
\end{eqnarray}
The Gram matrices in $\Xspace$ and $\Yspace$ are
\begin{equation}\label{eq:GX_GY}
G_{\Xspace}:=(M_{\rho}^{\Xspace})^2\blockdiag\big(G_{\mathbb{V}},\,G_{\mathbb{V}},\,G_{\mathbb{W}}\big)
\quad\mbox{and}\quad 
G_{\Yspace}:=\blockdiag\big(G_{\mathbb{W}},\,G_{\mathbb{W}},\,G_{\mathbb{V}}\big). 
\end{equation}
These matrices have the following properties: 

\begin{lemma}\label{lemm:prop_discr_DO}
\begin{enumerate}
\item\label{it:divcurl} $\divi{\kv}\curli{\kv}^{\#} = \divs_{\kv}\curls_{\kv}=0$
and $\curli{\kv}\gradi{\kv}=\curls_{\kv}\grads_{\kv}= 0$. 
\item\label{it:NDz} $\Mc{N}(D_{z}^{\mathbb{W}})=\{0\}$
\item\label{it:curl_adj} $\curli{\kv}^{\#}$ is the adjoint of $\curli{\kv}$ with respect to 
the Gram matrices $G_{\Xspace}$ and $G_{\Yspace}$, i.e.\ 
$G_{\Xspace}\curli{\kv}^{\#}= (G_{\Yspace}\curli{\kv})^*$, 
and similarly $G_{\mathbb{V}}\divi{\kv}=-(G_{\Xspace}\gradi{\kv})^*$.
\item\label{it:ortho} With respect to the Gram matrix 
$G_{\Xspace}$ we have the orthogonal 
decomposition 
\begin{equation}\label{eq:ortho_discrete}
\fl
\Xspace_{\kv} = \Mc{N}(\divi{\kv})\oplus \Mc{N}(\curli{\kv})\quad\mbox{and}\quad 
 \Mc{N}(\divi{\kv})= \Mc{R}(\curli{\kv}^{\#}) \qquad
\mbox{for }\kv\neq 0.
\end{equation}
\end{enumerate}
\end{lemma}

\begin{proof}
Part (\ref{it:divcurl}) can be verified by straightforward computations. \\
Part (\ref{it:NDz}) is also easy to see, and 
part (\ref{it:curl_adj}) follows from \eref{eq:adjointnessD}. \\
To show part (\ref{it:ortho}) we first demonstrate that 
\begin{equation}\label{eq:discrete_nullspaces}
\Mc{N}(\curli{\kv}) \cap \Mc{N}(\divi{\kv}) = \{0\}
\end{equation}
Let $\bsol_\kv= (\sol_\kv^{\rm x},\sol_\kv^{\rm y},\sol_\kv^{\rm z})\in 
\Mc{N}(\curli{\kv}) \cap \Mc{N}(\divi{\kv})$.  We only treat 
the case $k_{\rm x}\neq 0$ as the case $k_{\rm y}\neq 0$ is analogous. The last line in 
$\curli{\kv}\bsol_\kv=0$ implies that 
\begin{equation}\label{eq:auxlemm1}
k_{\rm x} M_\rho^\mathbb{V} \sol_\kv^{\rm y} = k_{\rm y} M_\rho^\mathbb{V} \sol_\kv^{\rm x}.
\end{equation}
Together with the relation $\divi{\kv}\bsol_\kv=0$ this yields 
\begin{equation}\label{eq:auxlemm2}
ik_{\rm x}D_{z}^{\mathbb W} M_\rho^{\mathbb W} \sol_\kv^{\rm z} = k_{\rm x}^2 M_\rho^\mathbb{V} \sol_\kv^{\rm x} 
+ k_{\rm x}k_{\rm y} M_\rho^\mathbb{V} \sol_\kv^{\rm y}
= |\kv|^2 M_\rho^\mathbb{V} \sol_\kv^{\rm x}.
\end{equation}
From the second line in $\curli{\kv}\bsol_\kv=0$ we obtain 
$ik_{\rm x} M_\rho^\mathbb{W} \sol_\kv^{\rm z}= D_{z}^{\mathbb V} M_\rho^{\mathbb V} \sol_\kv^{\rm x}$, so 
\[
D_{z}^{\mathbb W} D_{z}^{\mathbb V} M_\rho^{\mathbb V} \sol_\kv^{\rm x} = |\kv|^2 M_\rho^\mathbb{V} \sol_\kv^{\rm x}.
\]
Together with \eref{eq:adjointnessD} we find that 
$\left({D_{z}^{\mathbb V}}^* G_{\mathbb{W}}^* D_{z}^{\mathbb V} + 
|k|^2 G_{\mathbb{V}}^{-1}\right) M_\rho^\mathbb{V} \sol_\kv^{\rm x}=0$. 
Since the matrix on the left hand side is strictly positive definite, it follows 
that $\sol_\kv^{\rm x}=0$. 
Now it follows from part (\ref{it:NDz}),
\eref{eq:auxlemm1}, \eref{eq:auxlemm2} and $k_{\rm x}\neq 0$ that 
 $\sol_\kv^{\rm y}=0$ and $\sol_\kv^{\rm z}=0$, completing the proof of 
\eref{eq:discrete_nullspaces}. 

From parts (\ref{it:divcurl}) and (\ref{it:curl_adj}) we obtain
\[
\Mc{N}(\curli{\kv})^{\perp} = \Mc{R}(\curli{\kv}^{\#})
\subset \Mc{N}(\divi{\kv})
\]
with orthogonality with respect to the inner product generated by $G_{\mathbb{X}}$. 
Together with \eref{eq:discrete_nullspaces} this implies \eref{eq:ortho_discrete}.
\end{proof}

\begin{remark}\label{rem:k0}
Let us discuss the case $\kv=0$. 
We claim that in analogy to the continuous situation we have 
\begin{equation}\label{eq:null_intersect}
\fl
\Mc{N}(\curli{\mathbf{0}}) \cap \Mc{N}(\divi{\mathbf{0}}) 
= \braces{\paren{c_{\rm x}(M_{\rho}^{\mathbb{V}})^{-1}e,c_{\rm x}(M_{\rho}^{\mathbb{V}})^{-1}e,0}:
c_{\rm x},c_{\rm y}\in\mathbb{C}}
\end{equation}
where $e\in\mathbb{V}$ is the vector with all entries equal to $1$. 
In fact, for $\bsol_{\mathbf{0}}\in \Mc{N}(\curli{\mathbf{0}}) \cap 
\Mc{N}(\divi{\mathbf{0}})$ it follows follows from part (\ref{it:NDz}) and 
$\divi{{\bf 0}}\bsol_{\bf 0}=0$ that $\sol_{\bf 0}^{\rm z}=0$. Note 
that $\Mc{N}(D_{z}^{\mathbb W})=\{c(M_{\rho}^{\mathbb{V}})^{-1}e:c\in\mathbb{C}\}$. 
Now  \eref{eq:null_intersect} follows from $\curli{\mathbf{0}}\bsol_{\bf 0}=0$.
\end{remark}

The projection matrices onto $\Mc{N}(\curli{\kv})$ and $\Mc{N}(\divi{\kv})$  
can be computed using a QR-decomposition of $G_{\mathbb{X}}^{1/2}\curli{\kv}^{\#}$:
\begin{equation}\label{eq:defiPk}
G_{\mathbb{X}}^{1/2}\curli{\kv}^{\#} = (Q_\kv \;\tilde{Q}_\kv)\left(
\scriptsize{\begin{array}{c}R_\kv \\ 0\end{array}}\right), \qquad 
P_{\kv}:= G_{\Xspace}^{-1/2}Q_{\kv}Q_{\kv}^* G_{\Xspace}^{1/2},\qquad \kv\neq 0.
\end{equation}
Here $R_\kv$ has full row rank $p$, $[Q_\kv \;\tilde{Q}_\kv]$ is unitary, and 
$Q_\kv$ has $p$ columns. We summarize the properties of $P_\kv$:
\begin{lemma}\label{lemm:projections}
Let $\kv\neq 0$. Then 
$P_{\kv}$ is a projection onto $\Mc{N}(\divi{\kv})$ (i.e.\ $P_{\kv}^2=P_{\kv}$ and 
$\Mc{R}(P_{\kv})= \Mc{N}(\divi{\kv})$), and $I-P_{\kv}$ is a projection onto 
$\Mc{N}(\curli{\kv})$. $P_{\kv}$ is orthogonal both with respect to the inner product 
induced by $G_{\Xspace}$ (i.e.\ $P_{\kv}^*G_{\Xspace}= G_{\Xspace}P_{\kv}$) and 
the semi-definit inner product induced by the (Hermitian) Gram matrix 
\[
G^{H^1}_{\kv,\rho} =  G_{\Xspace}\curli{\kv}^{\#} G_{\Yspace} \curli{\kv} 
-G_{\Xspace}\gradi{\kv}G_{\mathbb{V}}\divi{\kv} 
\]
(i.e.\ $P_{\kv}^*G^{H^1}_{\kv,\rho}= G^{H^1}_{\kv,\rho}P_{\kv}$). 
\end{lemma}

\proof
The identity $P_{\kv}^*G_{\Xspace}=  G_{\Xspace}^{1/2}Q_{\kv}Q_{\kv}^* G_{\Xspace}^{1/2}
= G_{\Xspace}P_{\kv}$ is obvious from the definition. 
We have $P_{\kv}^2 =  G_{\Xspace}^{-1/2}Q_{\kv}(Q_{\kv}^* Q_{\kv})Q_{\kv}^* G_{\Xspace}^{1/2}
=P_{\kv}$, so $P_{\kv}$ is a projection, which implies that $I-P_{\kv}$ is a projection as well. 
Using Lemma \ref{lemm:prop_discr_DO}, parts (\ref{it:curl_adj}) and (\ref{it:divcurl}) we obtain 
\[
\Mc{R}(P_{\kv}) = \Mc{R}(G_{\Xspace}^{-1/2}Q_\kv) 
= \Mc{R}(\curli{\kv}^{\#})
= \Mc{N}(\divi{\kv}). 
\]
Moreover, $\Mc{R}(I-P_{\kv}) = \Mc{R}(P_{\kv})^{\perp} = \Mc{N}(\divi{\kv})^{\perp} 
= \Mc{N}(\curli{\kv}^{\#})$ using \ref{lemm:prop_discr_DO}(\ref{it:ortho}) and the 
self-adjointness of $P_{\kv}$ in $\Xspace$. 
By Lemma \ref{lemm:prop_discr_DO}(\ref{it:curl_adj}) we have 
$G^{H^1}_{\kv,\rho} = \curli{\kv}^*G_{\Yspace}^2\curli{\kv}+\divi{\kv}^*G_{\mathbb{V}}^2\divi{\kv}$, 
so $G^{H^1}_{\kv,\rho}$ is Hermitian and positive semi-definite. 
Moreover, since $\divi{\kv} P_{\kv}=0$ we have 
\begin{eqnarray*}
\fl
\lefteqn{P_{\kv}^*G^{H^1}_{\kv,\rho} 
= \paren{P_{\kv}^*G_{\Xspace}^{1/2}}\paren{G_{\Xspace}^{1/2}\curli{\kv}^{\#}}G_{\Yspace}\curli{\kv} 
= \paren{G_{\Xspace}^{1/2}Q_\kv Q_\kv^* }\paren{Q_\kv R_\kv}G_{\Yspace}\curli{\kv}}\\
&=& G_{\Xspace}^{1/2}\paren{Q_\kv R_\kv}G_{\Yspace}\curli{\kv}
= G_{\Xspace}\curli{\kv}^{\#}G_{\Yspace}\curli{\kv}
= \curli{\kv}^*G_{\Yspace}^2\curli{\kv}.
\end{eqnarray*}
Since the right hand side of this equation is Hermitian, so is the left hand side, which implies 
$P_{\kv}^*G^{H^1}_{\kv,\rho} = G^{H^1}_{\kv,\rho}P_\kv$. 
\qed

For $\kv=0$ we define $P_\kv$ as follows:
\begin{equation}\label{eq:defiP0}
\left(\begin{array}{c}
G_{\mathbb{X}}^{1/2}\curli{0}^{\#} \\
((M_{\rho}^{\mathbb{V}})^{-1}e,0,0) \\
(0,(M_{\rho}^{\mathbb{V}})^{-1}e,0) \\
\end{array}\right)
= (Q_0 \;\tilde{Q}_0)\left(
\begin{array}{c}R_0 \\ 0\end{array}\right), \qquad 
P_{0}:= G_{\Xspace}^{-1/2}Q_{0}Q_{0}^* G_{\Xspace}^{1/2},
\end{equation}
see Remark \ref{rem:k0}. In this case $P_0$ is the orthogonal projection 
onto $\Mc{N}(\divi{0})\oplus \braces{\paren{c_{\rm x}(M_{\rho}^{\mathbb{V}})^{-1}e,c_{\rm x}(M_{\rho}^{\mathbb{V}})^{-1}e,0}:
c_{\rm x},c_{\rm y}\in\mathbb{C}}$. 
%

\subsection{Implementation of the Pinsker estimator with mass conservation constraint}
Let us recall of the definition of the Generalized Singular Value 
Decomposition GSVD (see \cite{VAN76}): 
Let $A \in \mathbb{R}^{m,n}$ and $L  \in \mathbb{R}^{q,n}$ be matrices with $m \geq n$ and $\mathrm{rank}(L) = p$. Then there exist unitary matrices $U \in \mathbb{R}^{m,m}$ and  
$V \in \mathbb{R}^{q,q}$ and an invertible matrix $X \in \mathbb{R}^{n,n}$ such that
 \begin{equation}
  A = U S X^{-1} \quad \mbox{ and } \quad L = V C X^{-1}
 \end{equation}
where $S = \diag(s_1, ...,s_n)\in\mathbb{R}^{m\times n}$ and 
$C = \diag(c_1, ..., c_{\min(q,n)})\in\mathbb{R}^{q\times n}$ with 
$1 \geq c_1 \geq \ldots \geq c_p > c_{p+1}=\dots=0$. 
The generalized singular values $\sigma_i$ of $(A,L)$ are $\sigma_i = s_i / c_i$ for $i = 1, \ldots, p$, and the generalized right singular vectors of 
$(A,L)$ are the first $p$ columns $x_1,\dots,x_p$ of $X$. They satisfy the orthogonality relations 
$x_j^*L^*Lx_k=c_j^2\delta_{j,k}$ for $j,k\in\{1,\dots,n\}$. 
If $L = I$ then the GSVD and the SVD coincide (except for that the ordering of the singular values).



\medskip
We will set $A:= \Lambda_\kv^{-1/2}K_{\kv}P_{\kv}$ with $P_{\kv}$ defined in \Eq{defiPk} and 
$L:=\paren{\begin{array}{c}G_{\Yspace}^{1/2}\curli{\kv} 
\\ G_{\mathbb{V}}^{1/2}\divi{\kv}\end{array}}$. 
For $\kv\neq 0$ this yields vectors $x_1,\dots,x_{\dim(\Xspace_{\kv})}$, $v_1,\dots,v_{\dim(\Xspace_{\kv})}$, 
and $u_1,\dots,u_p$ and numbers $s_1,\dots,s_p,c_1,\dots,c_p>0$ such that 
\begin{eqnarray*}
\Lambda_\kv^{-1/2}K_{\kv}P_{\kv}x_j &=& s_ju_j,\qquad j=1,\dots,p\\
Lx_j &=& c_jv_j,\qquad   j=1\dots,\dim(\Xspace_{\kv}).
\end{eqnarray*}
For $j\leq p$ we have 
\[
x_j\in \Mc{N}(\Lambda_\kv^{-1/2}K_{\kv}P_{\kv})^{\perp}\subset \Mc{N}(P_{\kv})^{\perp}=\Mc{N}(\divi{\kv})
\]
with orthogonality w.r.t.\ the $L$-induced inner product and 
$x_j^{*}G^{H^1}_{\kv,\rho}x_j=\|Lx_j\|^2 = c_j^2$. Therefore $x_j/c_j$ are the 
singular vectors of $A$ w.r.t.\ this inner product, and the $\kv$-th Fourier coefficient 
of the Pinsker estimator is 
\[
\fl
W_\kv\tau_\kv 
= \sum_{j=1}^p \frac{\max(1-\overline{\kappa}a_j,0)}{\sigma_j}
\langle u_{j},\Lambda_\kv^{-1/2}\tau_\kv\rangle \frac{x_j}{c_j} 
= \sum_{j=1}^p \frac{\max(1-\overline{\kappa}a_j,0)}{s_j}
\langle u_{j},\Lambda_\kv^{-1/2}\tau_\kv\rangle x_j.
\]


\begin{algorithm}[ht]
 \KwData{$\bullet$ kernels $K_\kv\in\mathbb{C}^{M\times N_a}$ and noise 
covariance matrices $\Lambda_\kv\in\mathbb{C}^{N_a\times N_a}$ for all frequencies $\kv$\;
$\bullet$ regularization parameter $\overline{\kappa}$\;}
 \KwResult{linear estimator $W_\kv$ for all frequencies $\kv$}
set up Gram matrices $G_{\mathbb{V}},G_{\mathbb{W}},G_{\mathbb{X}}$, and $G_{\mathbb{Y}}$ 
(eqs.~(\ref{eq:GV_GW}), (\ref{eq:GX_GY}))\;
 \For{$\kv \in [-N_k/2,\dots,N_k/2-1]^2$}{
set up matrices $\divi{\kv}$, $\curli{\kv}$  
and $P_\kv$  (eqs.~(\ref{eq:divrho_curlrho}), (\ref{eq:defiPk}),
(\ref{eq:defiP0})) \;
  $[U_\kv,X_\kv, V_\kv, s_\kv,c_\kv] = \mathrm{gsvd}\left(\Lambda_\kv^{-1/2} K_\kv P_\kv, 
\paren{\begin{array}{c}G_{\Yspace}^{1/2}\curls_{\rho,\kv} 
\\ G_{\mathbb{V}}^{1/2}\divs_{\rho,\kv}\end{array}} \right)$ \;
\hspace*{2ex}($U_\kv,X_kv,V_\kv$ are matrices with columns $u_{\kv,j},x_{\kv,j},v_{\kv,j}$\;
\hspace*{2ex}$s_\kv$, $c_\kv$ are vectors with entries $s_{\kv,j}$, $c_{\kv,j}$)\;
}
Find bijective ordering $l: [-\frac{N_k}{2},\dots,\frac{N_k}{2}\!-\!1]^2
\times \{1,\dots,M\}
\to N_k^2N_z$ such that $\frac{s_{\kv,j}}{c_{\kv,j}}\geq 
\frac{s_{\tilde{\kv},\tilde{j}}}{c_{\tilde{\kv},\tilde{j}}}$ if 
$l(\kv,j)\leq l(\tilde{\kv},\tilde{j})$,
$c_{\kv,j},c_{\tilde{\kv},\tilde{j}}>0$. $l(\kv,j)\geq l(\tilde{\kv},\tilde{j})$ 
if $c_{\kv,j}\!=\!0$ and $c_{\tilde{\kv},\tilde{j}}\!>\!0$\;
\For{$\kv \in [-N_k/2,\dots,N_k/2-1]^2$}{
$p_{\kv}=\max\{j:c_{\kv,j}>0\}$\;
\For{$j=1,\dots,p_\kv$}{
$a_{\kv,j}:=l(\kv,j)^{1/3}$\;
$\lambda_{\kv,j}:= \max(1 - \overline{\kappa} a_{\kv,j}, 0)$ \;
}
 $W_\kv = X_\kv[:,1:p_\kv]  \diag\paren{\frac{\lambda_{\kv,1}}{s_{\kv,1}},\dots, 
\frac{\lambda_{\kv,p_{\kv}}}{s_{\kv,p_\kv}}}  
U_\kv[:,1:p_{\kv}]^*  \Lambda_\kv^{-1/2}$ \;
 }
 \caption{Pinsker algorithm with mass conservation constraint. 
}
\label{algo:pinsker_massConservation}
\end{algorithm}

\section{Numerical results} \label{sect:numerics}

In the following we will compare RLS, SOLA and Pinsker methods for recovering three-dimensional velocity fields from travel time measurements on the solar surface. 

To compare the different inversion methods on synthetic data,  we use the velocity model presented in \cite{DOM13} which reproduces an average supergranule. Supergranulation is a convection pattern with an average life time of about 1 day and a characteristic length of around $30\Mm$ that is observed at the surface of the Sun. A representation of the velocity field $\sol^{\mathrm{z}}$ and $\sol^{\mathrm{x}}$  is given in  \Figs{vzReconstructedNoDiv}{vxReconstructed} (top rows).  This velocity is built such that mass is conserved,  which explains the decrease of the amplitude with depth due to the strong density gradient.  These velocities are then convolved with the kernels, and noise is added according to \Eq{travelTime} in order to obtain travel time maps as shown in \Fig{travelTime}. 

\subsection{Reconstruction without mass conservation} \label{sect:numWithout}
In the RLS method we have chosen the regularization term as $H^1$ norm in horizontal 
and vertical directions, and in 
the Pinsker method the ellipsoid $\Theta$ was chosen according to 
\Eq{choice_ellipsoid}) to approximate a ball in the Sobolev space $H^1(V)$. 
The regularization parameters $\alpha$ and $\overline{\kappa}$ have been chosen 
by the discrepancy principle. Although the discrepancy principle performs 
poorly for high dimensional white noise (and is not even well-defined in the 
infinite-dimensional case), here the noise is sufficiently correlated for 
the discrepancy principle to work reasonably well. 
The SOLA weighting kernels are obtained by minimizing \Eq{SOLAmin} with a target function
\begin{equation*}
 \Mc{T}^{\beta, z_j;\alpha,z_l}(\rv) = \exp \left( \frac{\norm{\rv}^2}{2s_h^2} + \frac{\norm{z_j-z_l}^2}{2s_v^2} \right) \delta_{\alpha \beta},
\end{equation*}
where $s_h$ and $s_v$ determine the localization of the averaging kernels in the horizontal and vertical directions. As usual we added a constraint for 
$\kv=0$ via Lagrange multipliers to ensure that the integrals over the averaging 
kernels for horizontal velocities are $1$. This is not possible for the vertical 
velocities since constant vertical flows are in the nullspace of $K$. 
To allow a fair comparison with RLS and Pinsker, 
we did not impose a strong additional penalty to suppress cross-talk as in 
\cite{SVA11} since we found that this induces a significant loss of resolution.

It is well-known in helioseismology that RLS and SOLA can reconstruct horizontal 
velocity components $\sol^{\rm x}, \sol^{\rm y}$ fairly well, but perform poorly for the 
reconstruction of vertical velocity components. 
\Fig{vzReconstructedNoDiv} shows the reconstruction of $\sol^{\rm z}$ for the different methods without mass conservation. As expected, the results for Tikhonov regularization are poor except close to the surface. The SOLA method is a bit better at larger depths and the Pinsker estimator leads to a clear improvement with almost correct reconstructions at $z_t = -3.5\Mm$ and a detection of a positive value of the velocity close to the center at $z_t = -5.5\Mm$. However, the amplitudes of the reconstructed velocities both at $-3.5\Mm$ and 
in particular at $-5.5\Mm$ are too small.

\begin{figure}[ht]
\centering
 \includegraphics[width=0.7\linewidth]{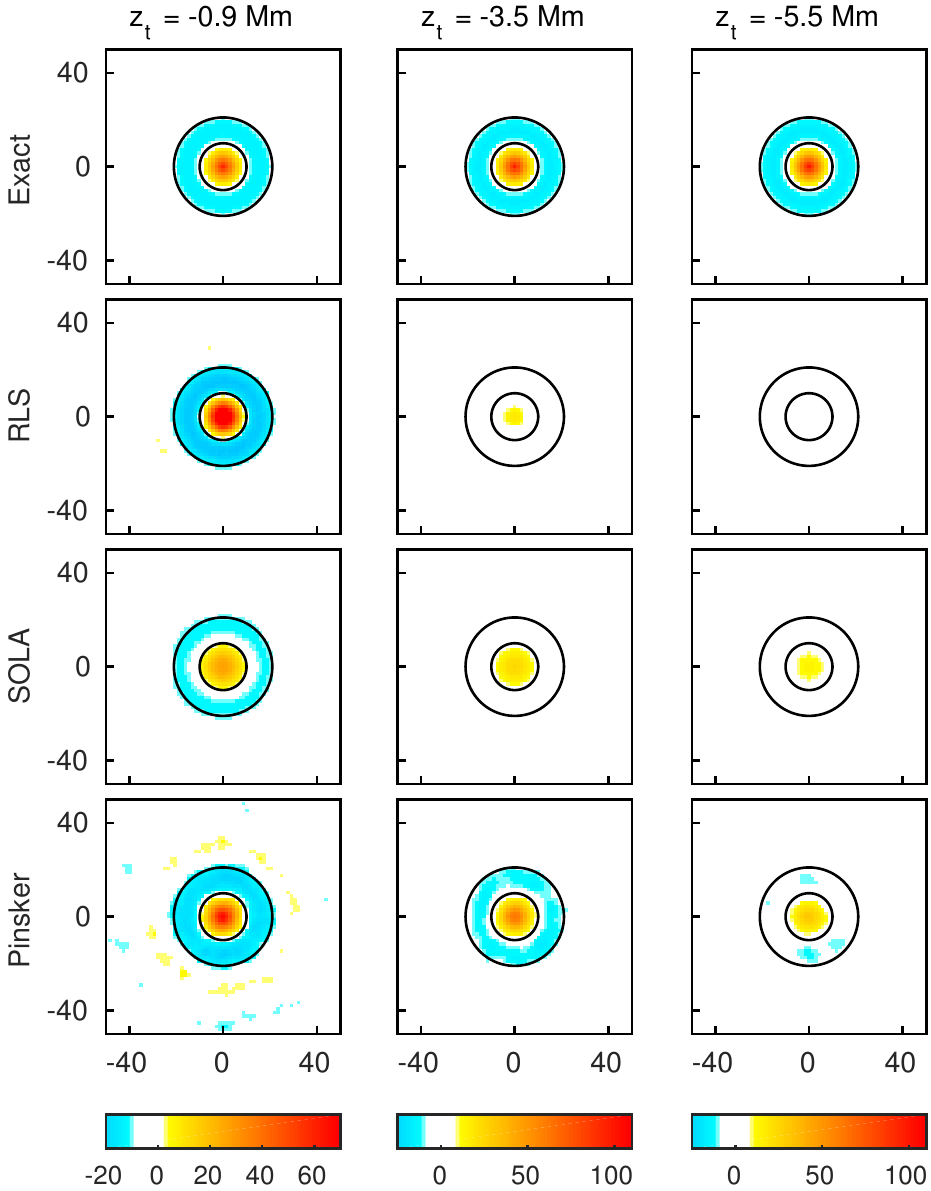}
 \caption{Vertical velocities $\sol^{\mathrm{z}}(x,y,z_{\rm t})$ in m/s of a  
supergranule model from Ref. \cite{DOM13}  (top) and their reconstructions with  
RLS ($2^\textrm{nd}$ row), SOLA ($3^\textrm{rd}$ row) and Pinsker (bottom)  at three different depths $z_{\rm t} \in \{ -0.9\Mm, -3.5\Mm,-5.5\Mm\}$.
The circles at radii $10\Mm$ and $20\Mm$ represent zero level lines of the 
exact solution.}
\label{fig:vzReconstructedNoDiv}
\end{figure}

To understand the difficulties of RLS with the reconstruction of $\sol^{\rm z}$, 
we look at the depth localisation of the averaging kernels. 
To compare the different estimators $W^{\beta}$ for some velocity component $\sol^{\beta}$, 
$\beta\in\{\mathrm{x,y,z}\}$, we choose the parameters in these methods such that the variance 
$\EE[|(W^{\beta}n)(\rv,z_{\rm t})|^2]$ at the target depth $z_{\rm t}$ has the 
same value for all the methods. (Due to translation invariance of the noise covariance 
structure this value is independent of $\rv$.).  
Then we compare the corresponding averaging 
kernels $\Mc{K}^{\beta,z_{\rm t};\alpha,z_j}$ describing the bias 
(see Definition \ref{defi:averaging_kernel}). In \Fig{avgz}, 
we represented the horizontal $L^2$ norm of 
$\Mc{K}^{\mathrm{z},z_{\rm t};\mathrm{z},z_j}$  as a function of the depth $z_j$ for RLS, 
SOLA and Pinsker methods at two different target depths $z_{\rm t}$. One can see that the averaging kernel for the RLS method is mostly localized close to the surface rather than at the target depth. In contrast, the averaging kernel of Pinsker 
is much better localized at $z_t$, but still exhibits some sensitivity to the 
values close to the surface. Intermediately, the SOLA averaging kernel is localized at the correct depth, but is extremely broad, so the reconstruction of $v_z$ at the target depth $z_{\rm t}$ is greatly influenced by the other depths.

\begin{figure}[ht]
\centering
 \begin{tabular}{cc}
  \includegraphics[width=0.45\linewidth]{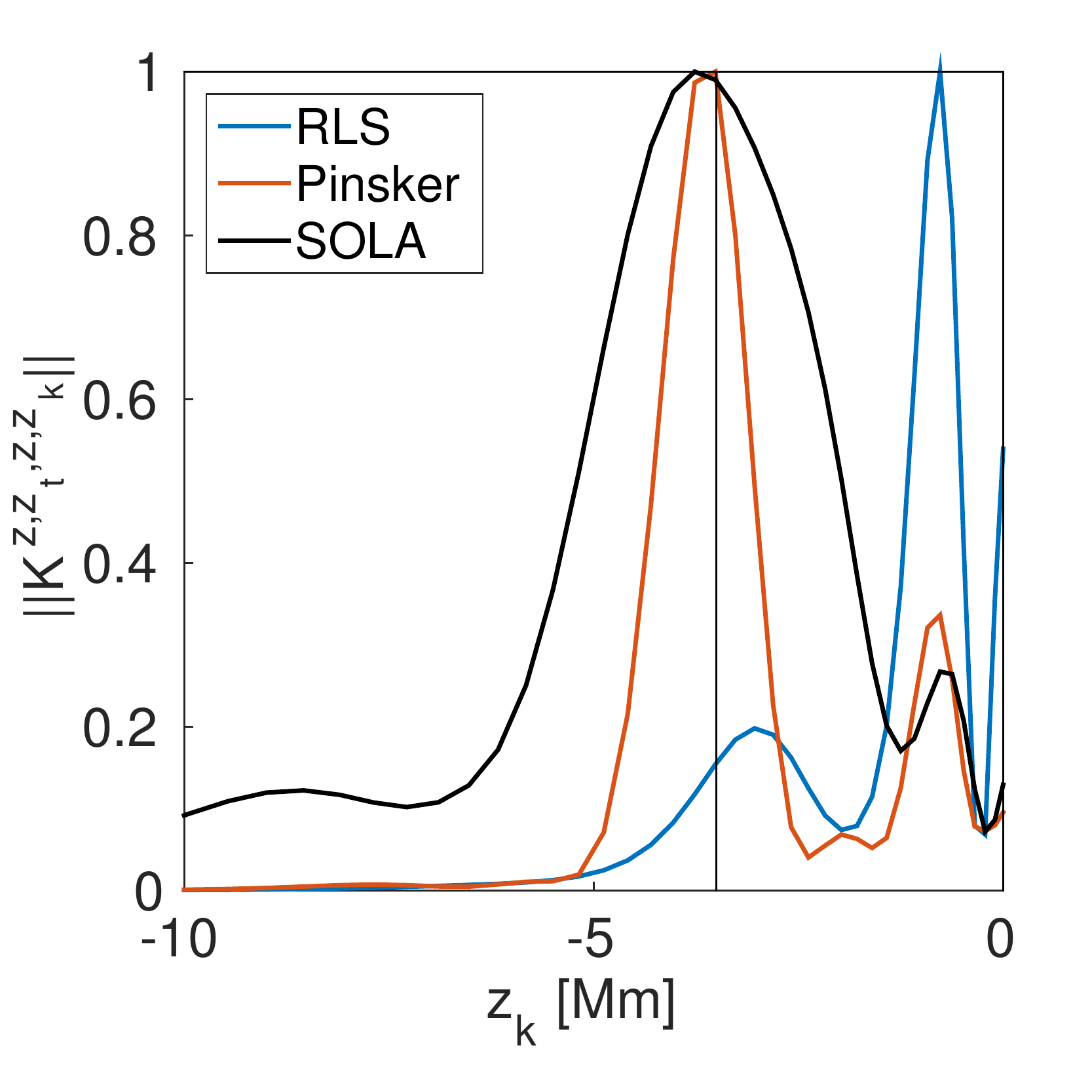} & \includegraphics[width=0.45\linewidth]{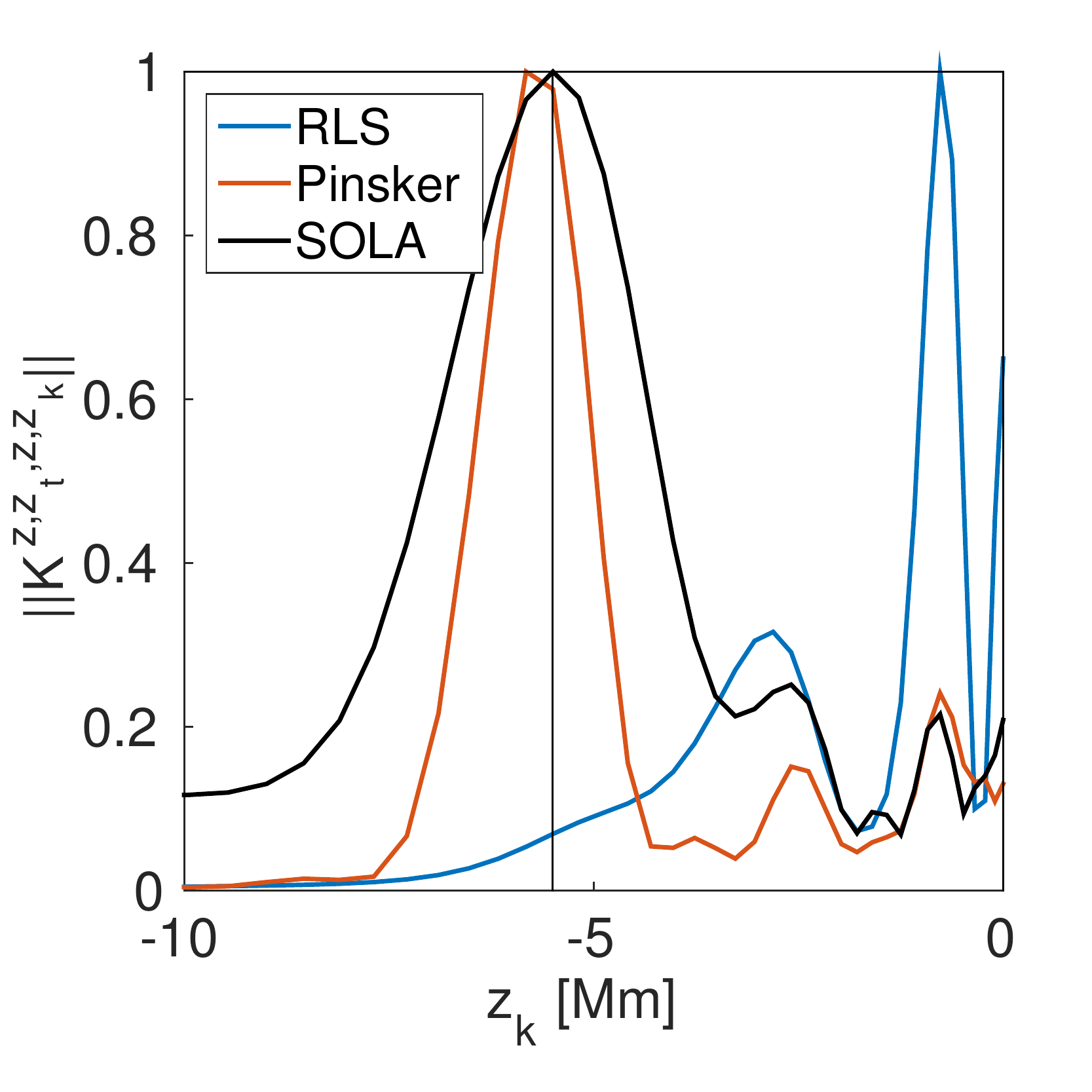}
 \end{tabular}
\caption{Horizontal norm of the averaging kernels $\Mc{K}^{\mathrm{z},z_{\rm t};\mathrm{z},z_k}$ 
as a function of $z_k$ for target depth $z_{\rm t} \in\{ -3.5\Mm, -5.5\Mm\}$ for the estimators 
RLS, SOLA and Pinsker. In each panel the 
regularization parameters are chosen such that 
for all three estimators $W^{\rm z}$ of the vertical velocity 
the standard deviation at the target depth $z_{\rm t}$ is 
$(\mathrm{Var}[(W^{\rm z}\tau)(\rv,z_{\rm t})])^{1/2} = 3\mathrm{m/s}$ 
for all $\rv$ 
for noise corresponding to an averaging time of 4 days.} 
\label{fig:avgz}
\end{figure}

The reconstructions of the horizontal velocity $\sol^{\mathrm{x}}$ by Tikhonov, SOLA and Pinsker methods are shown in \Fig{vxReconstructed}. As expected, all methods perform 
well. Surprisingly, from visual inspection Pinsker seems slightly less accurate 
than Tikhonov regularization at $z_{\rm t}=-5.5\Mm$.  

\begin{figure}[ht]
\centering
 \includegraphics[width=0.7\linewidth]{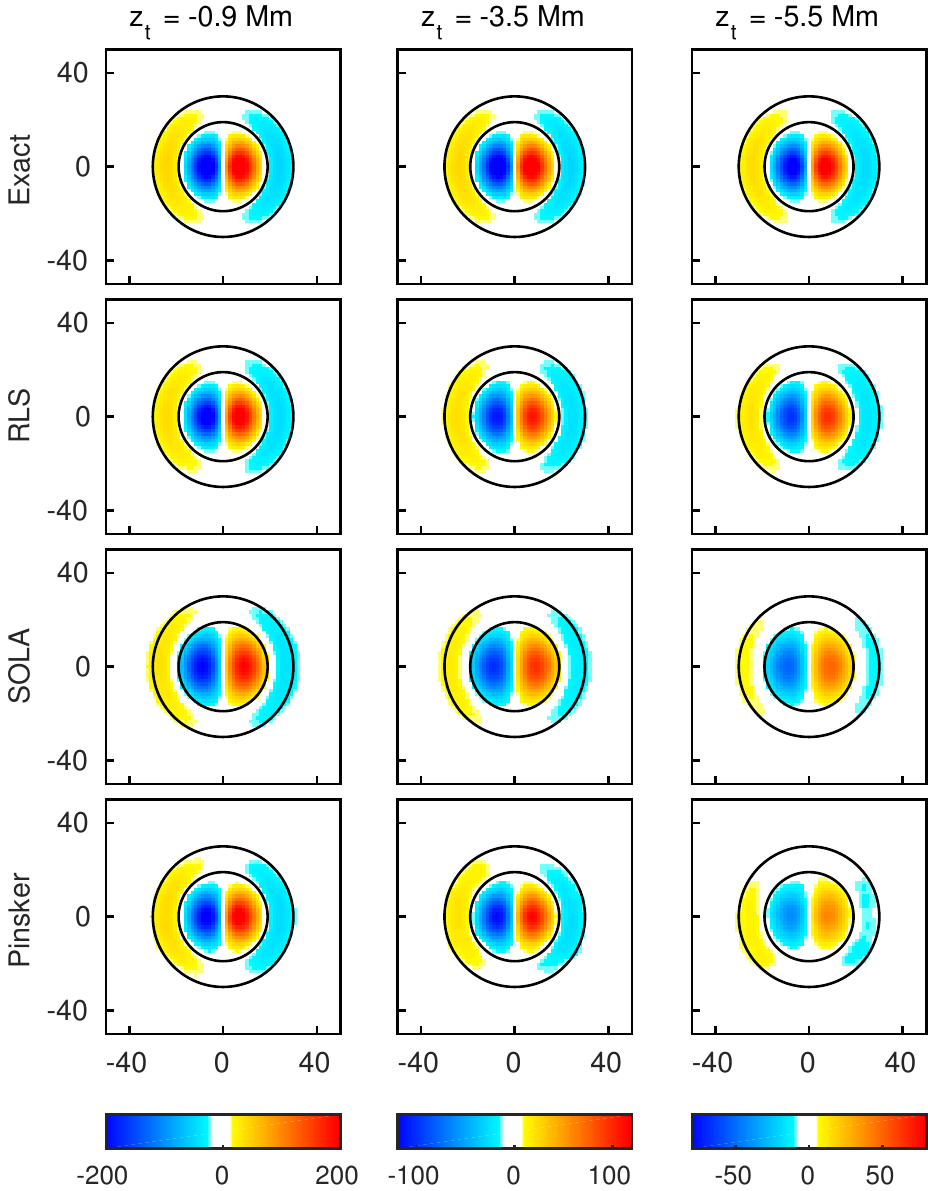}
\caption{Horizontal velocities $\sol^{\mathrm{x}}(x,y,z_{\rm t})$ in m/s of a  
supergranule model from Ref. \cite{DOM13} (top) and their reconstruction with different methods: RLS, SOLA and Pinsker 
(from top to bottom) at three different depths $z_{\rm t} = -0.9\Mm$ (left), $-3.5\Mm$ (middle), and $-5.5\Mm$ (right). The circles at $20\Mm$ and $30\Mm$ indicate the zero 
level lines of the exact horizontal velocity component.}
\label{fig:vxReconstructed}
\end{figure}

To get a better insight into the reconstructions, we can again look at the 
averaging kernels with the same choice of parameters as described above. 
\Fig{compZ} shows $\Mc{K}^{\mathrm{x},z_{\rm t};\mathrm{x},z_j}(x,0)$ as a function 
of $z_j$ and $x$ for three different depths 
$z_{\rm t}\in\{-0.9\Mm,-3.5\Mm,-5.5\Mm\}$ for the RLS, SOLA and Pinsker estimators. 

The differences between the three methods are the more pronounced the greater the target depth $z_{\rm t}$, i.e.\ the greater 
the ill-posedness. The Pinsker averaging kernels turn out to be the most localized, in particular in $z$ direction while
the SOLA averaging kernels are the least localized. 
RLS and Pinsker produce similar averaging kernels for the $\sol^{\rm x}$ estimators, 
which is consistent with the observed reconstructions. However, it is surprising that the reconstruction with the Pinsker method is not the best at $-5.5\Mm$ as 
the averaging kernels are the most localized. To explain this apparent inconsistency, we need to look at the cross-talk, i.e.\ how $\sol^\mathrm{y}$ and $\sol^\mathrm{z}$ influence the 
estimator of $\sol^{\mathrm{x}}$. \Fig{crosstalk} shows the averaging kernels $\Mc{K}^{\mathrm{x},z_{\rm t};\mathrm{y},z_j}$ and 
$\Mc{K}^{\mathrm{x},z_{\rm t};\mathrm{z},z_j}$ at a target depth of $z_{\rm t}=-3.5\Mm$. 
The cross talk is rather strong for Pinsker where the maximum value of the off-diagonal averaging kernels is only 
50\%  smaller than the maximum $\Mc{K}^{x,z_{\rm t};x,z_j}$, as opposed to around 
10\% for RLS and 5\% for SOLA. 

\begin{figure}[ht]
\begin{tabular}{cccc}
 & $z_{\rm t}=-0.9\Mm$ & $z_{\rm t}=-3.5\Mm$ & $z_{\rm t}=-5.5\Mm$ \\
 & \includegraphics[trim= 1cm 12cm 1.5cm 0cm, clip = true, width=0.3\linewidth]{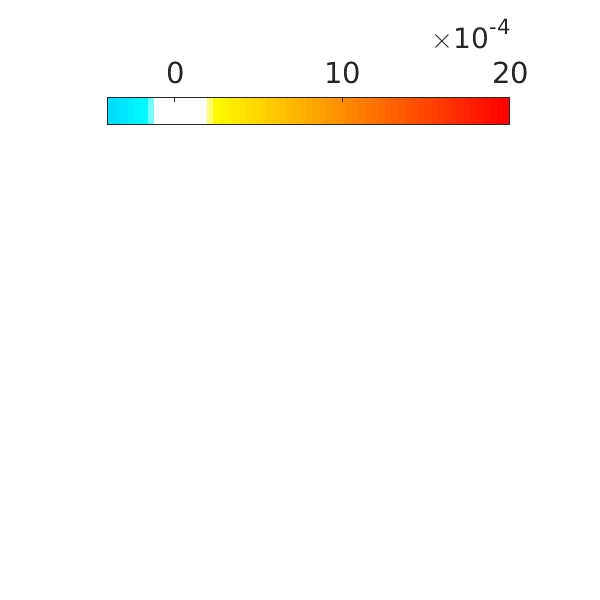} & \includegraphics[trim= 1cm 12cm 1.5cm 0cm, clip = true, width=0.3\linewidth]{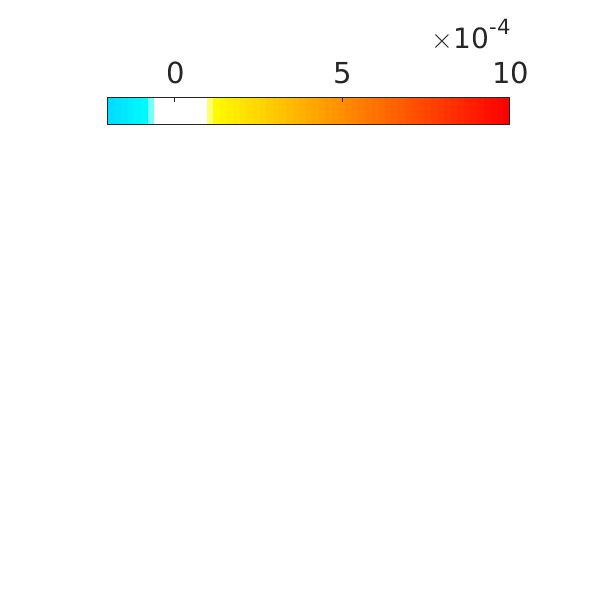} & \includegraphics[trim= 1cm 12cm 1.5cm 0cm, clip = true, width=0.3\linewidth]{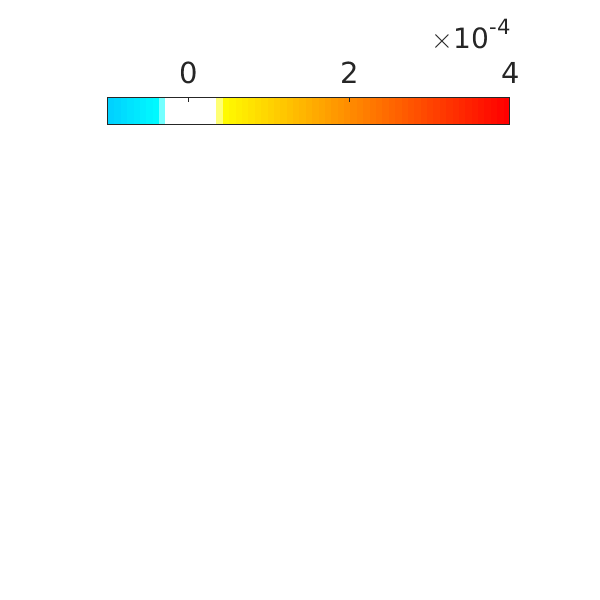} \\
 \raisebox{11ex}{\rotatebox{90}{SOLA}} & \includegraphics[width=0.3\linewidth]{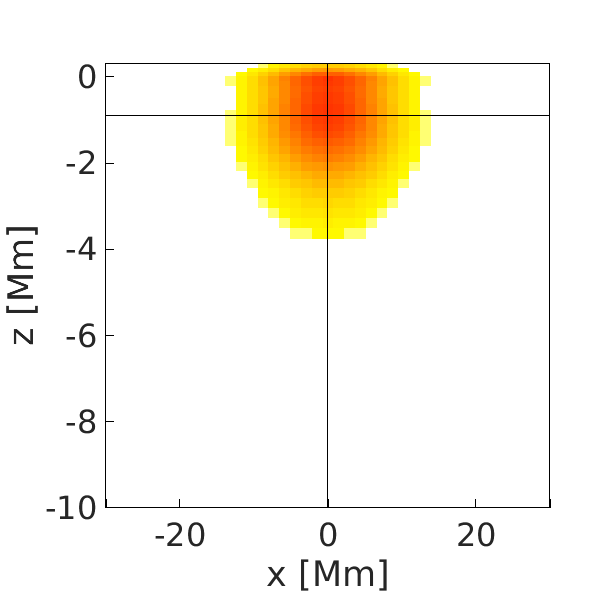} & \includegraphics[width=0.3\linewidth]{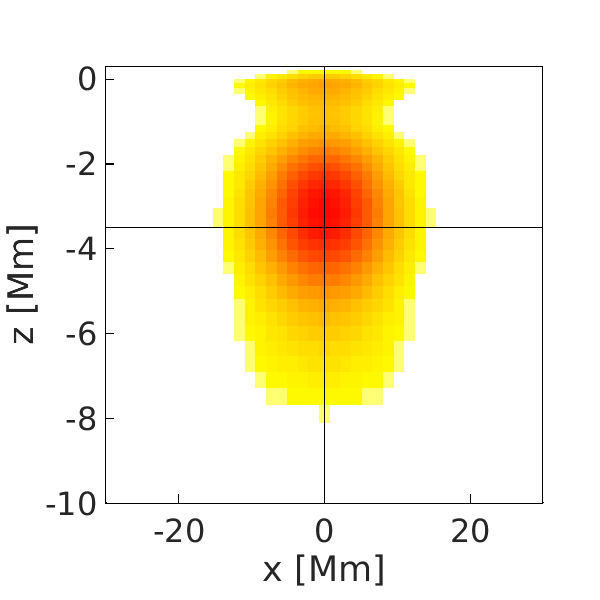} & \includegraphics[width=0.3\linewidth]{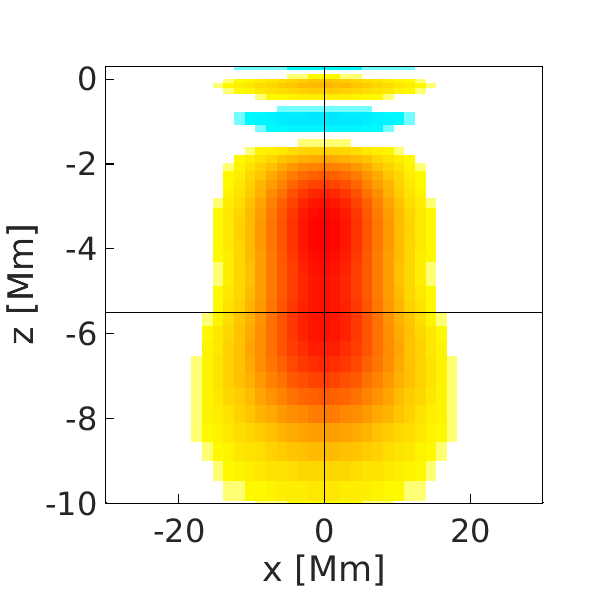} \\
 \raisebox{11ex}{\rotatebox{90}{RLS}} & \includegraphics[width=0.3\linewidth]{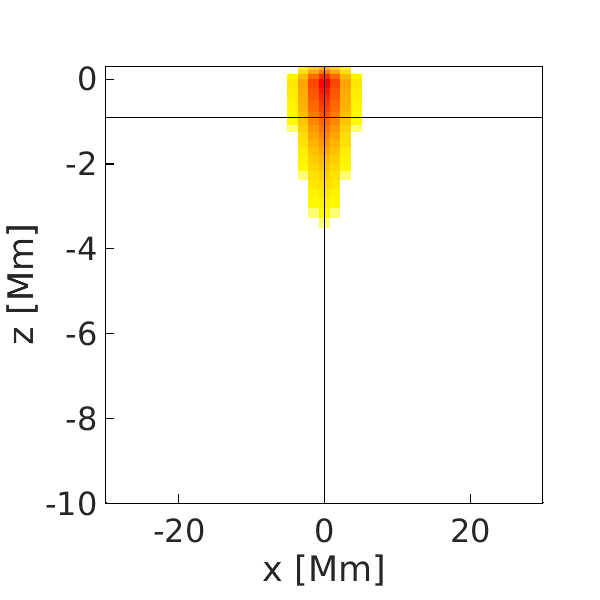} & \includegraphics[width=0.3\linewidth]{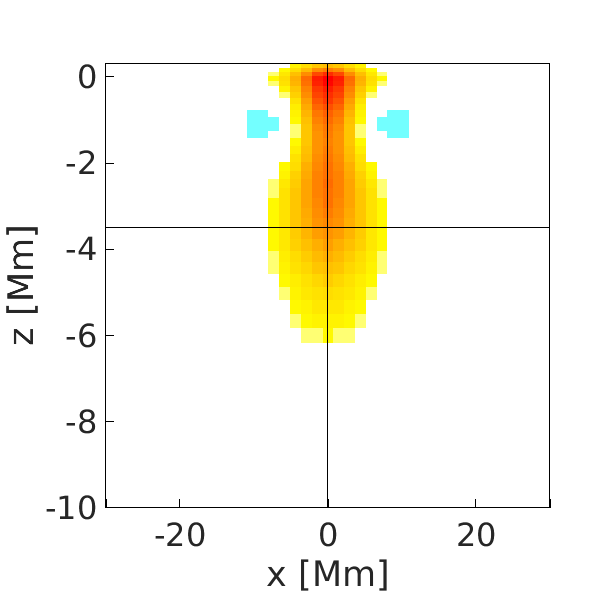} & \includegraphics[width=0.3\linewidth]{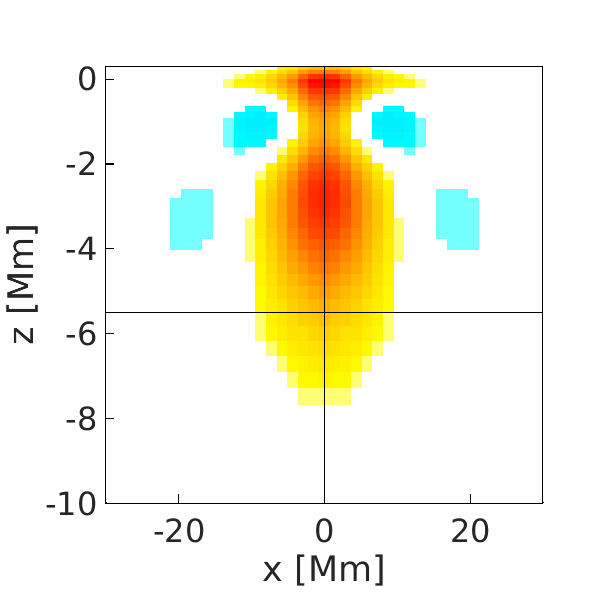} \\
 \raisebox{7ex}{\rotatebox{90}{Pinsker}} & \includegraphics[width=0.3\linewidth]{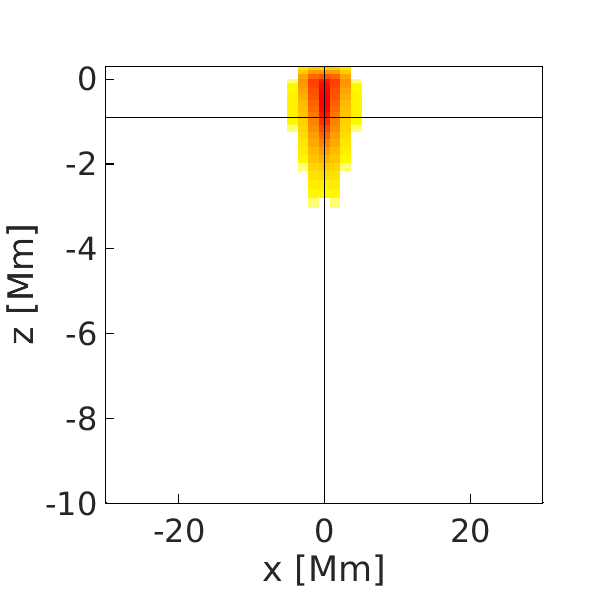} & \includegraphics[width=0.3\linewidth]{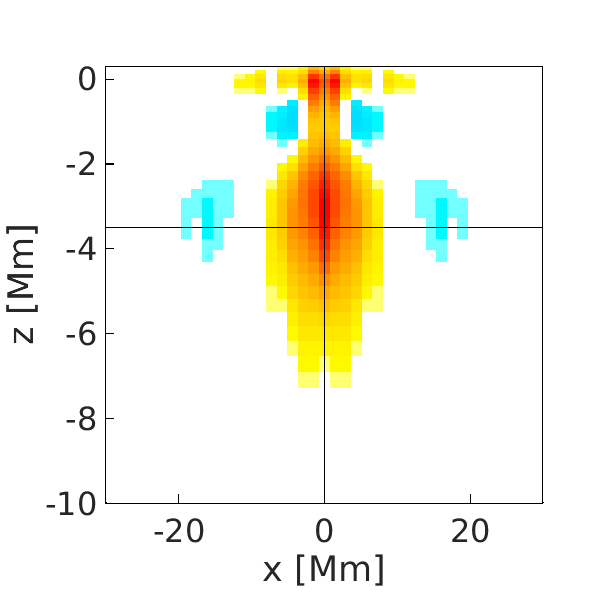} & \includegraphics[width=0.3\linewidth]{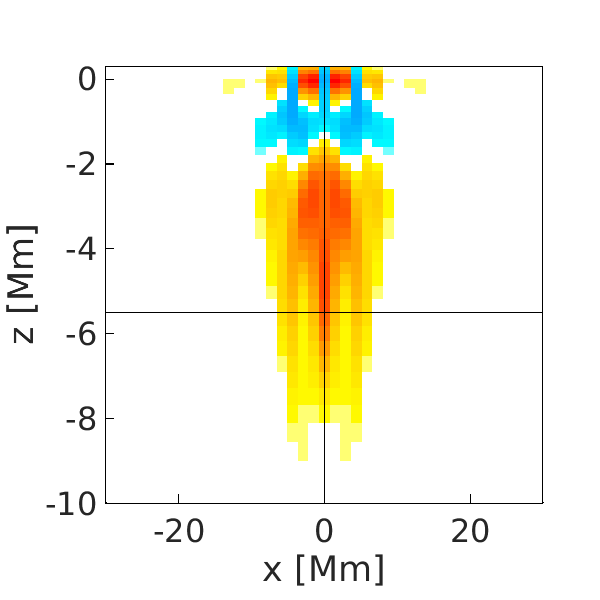} 
\end{tabular}
\caption{Averaging kernels $\Mc{K}^{\mathrm{x},z_{\rm t};\mathrm{x},z_j}(x,0)$ 
(characterizing the $\sol^{\rm x}$ influence on the bias of the $\sol^{\rm x}$ estimators) 
as functions of $x$ and $z_j$ for target depths 
$z_{\rm t} \in\{ -0.9\Mm, -3.5\Mm, -5.5\Mm\}$ for the estimators 
SOLA, RLS, and Pinsker. The variances of these estimators for each target 
depth are chosen to be of the same size.} \label{fig:compZ}
\end{figure}

\begin{figure}[ht]
\begin{center}
\begin{tabular}{ccc}
& $\sol^{\rm x}$-$\sol^{\rm y}$ crosstalk & $\sol^{\rm x}$-$\sol^{\rm z}$ crosstalk\\ 
\multicolumn{3}{c}{\includegraphics[trim= 1cm 12cm 2cm 0cm, clip = true, width=0.7\linewidth]{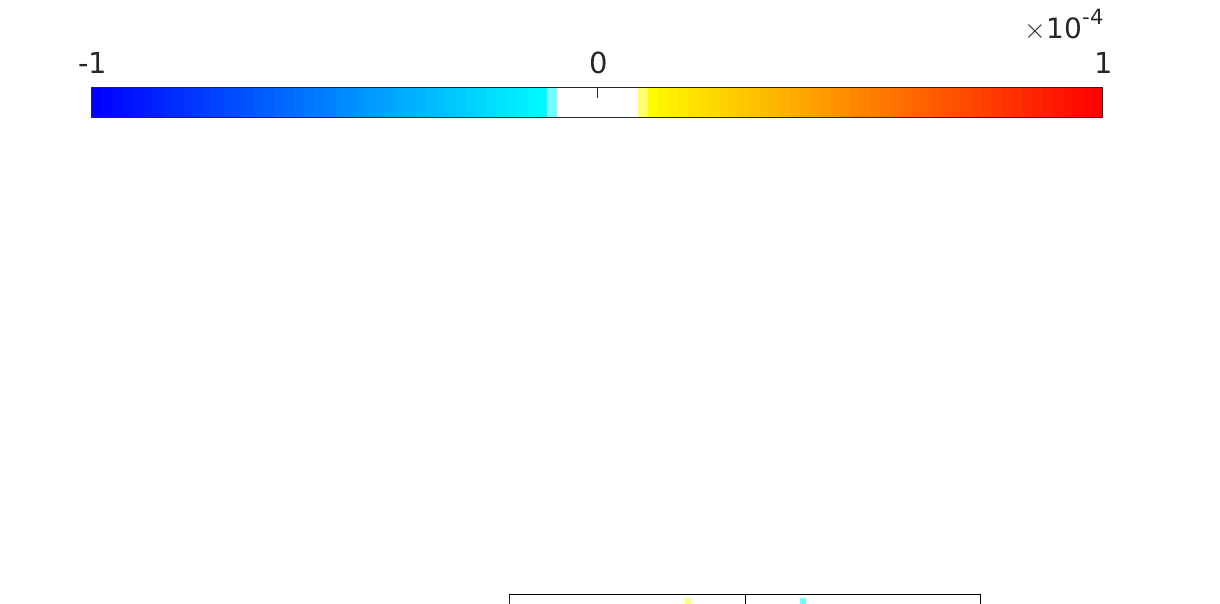}} \\
\raisebox{10ex}{\rotatebox{90}{SOLA}}
& \includegraphics[width=0.3\linewidth]{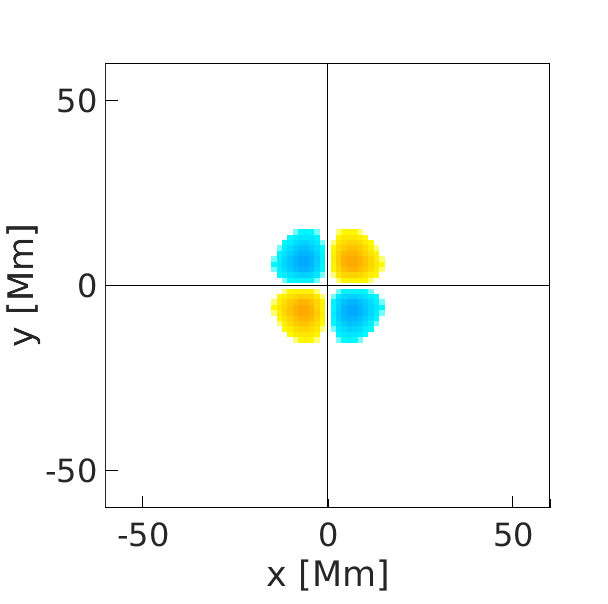} 
& \includegraphics[width=0.3\linewidth]{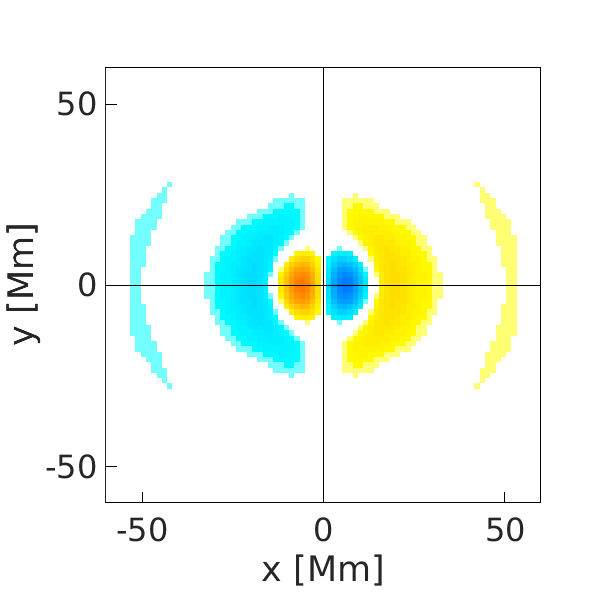} \\
\raisebox{10ex}{\rotatebox{90}{RLS}}
&\includegraphics[width=0.3\linewidth]{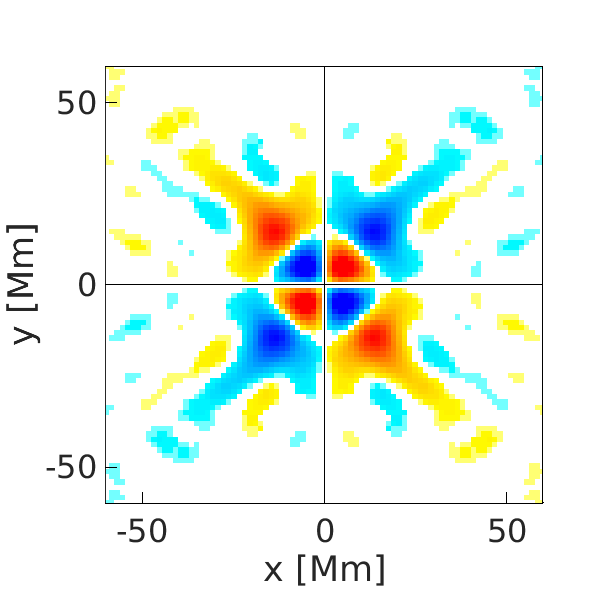} 
&\includegraphics[width=0.3\linewidth]{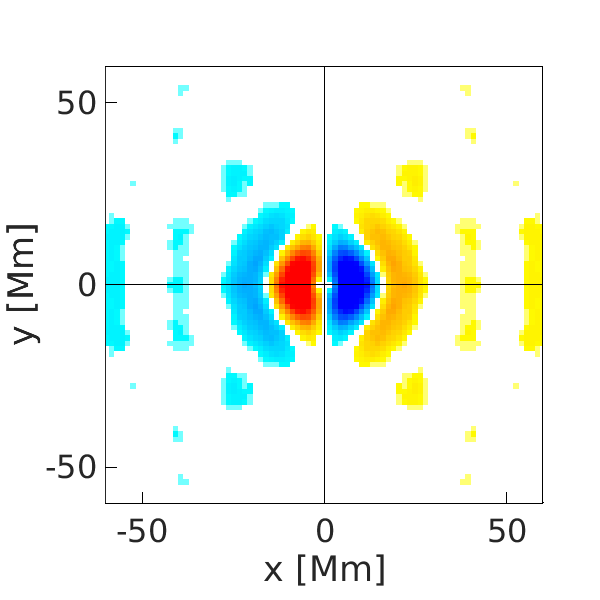} \\
\raisebox{8ex}{\rotatebox{90}{Pinsker}}
& \includegraphics[width=0.3\linewidth]{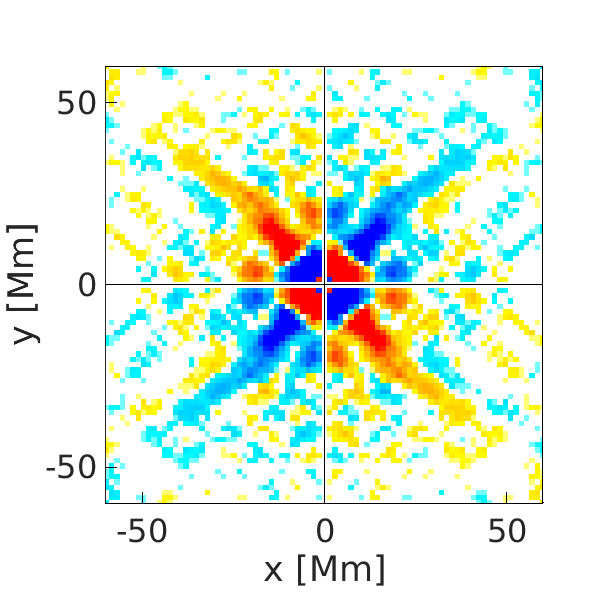} 
& \includegraphics[width=0.3\linewidth]{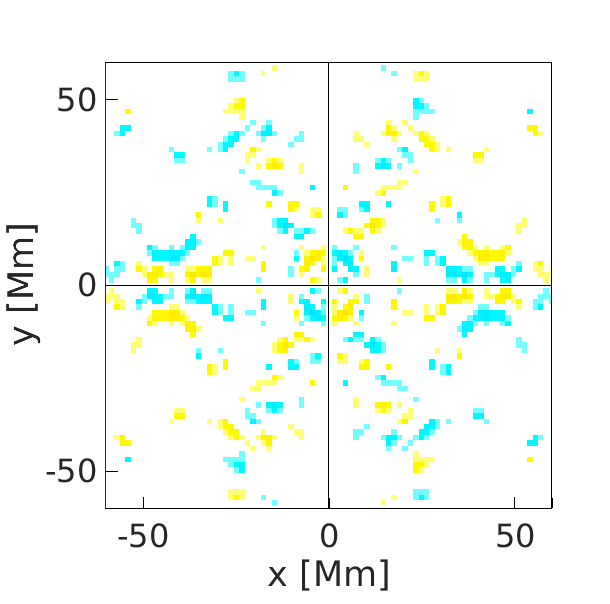}  
 \end{tabular}
\end{center}

 \caption{\label{fig:crosstalk}
Cross-talk averaging kernels $\Mc{K}^{\mathrm{x},z_{\rm t};\mathrm{y},z_j}$ and $\Mc{K}^{\mathrm{x},z_{\rm t};\mathrm{z},z_j}$ for $z_{\rm t} = -3.5\Mm$  for the SOLA, 
Tikhonov, and Pinsker methods. These kernels characterize the influence 
of the variables $\sol^\mathrm{y}$ and $\sol^\mathrm{z}$ on the bias of the $\sol^{\mathrm{x}}$-estimators.}
\end{figure}

\subsection{Incorporation of the divergence constraint}

\Fig{vzReconstructedDiv} shows the reconstruction of the vertical component of the velocity for the Tikhonov and Pinsker methods with mass conservation constraint. 
It underlines the importance of incorporating the constraint into the inversion 
process. The vertical velocity is now properly reconstructed by both methods.

To better compare all the methods, \Fig{vzCut} represents a cut of the 
vertical velocity at $x=0$ and $z_t \in \{-3.5\Mm, -5.5\Mm\}$. 
Incorporating mass conservation into Tikhonov leads 
to a quite good reconstruction with an amplitude of about 70\% of the true one. 
Finally, Pinsker with mass conservation is almost perfect at the depths up to $-5.5\Mm$ with correct shape and amplitude.  
 

\begin{figure}[ht]
\centering
 \includegraphics[width=0.7\linewidth]{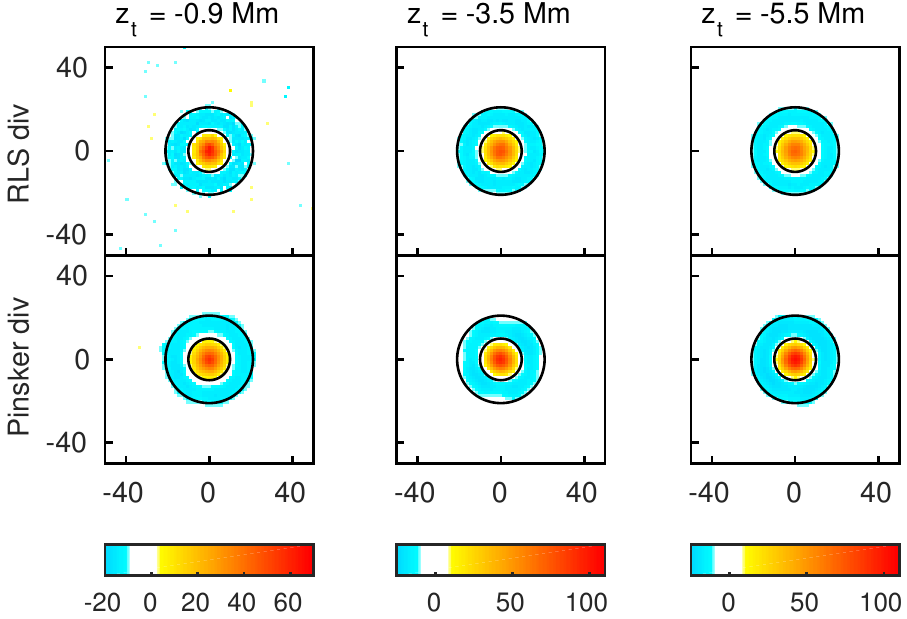}
 \caption{Reconstructions for $v_z$ as in \Fig{vzReconstructedNoDiv} but by imposing mass conservation (top: RLS, bottom: Pinsker).}
\label{fig:vzReconstructedDiv}
\end{figure}

\begin{figure}[ht]
 \begin{tabular}{cc}
 \includegraphics[width=0.45\textwidth]{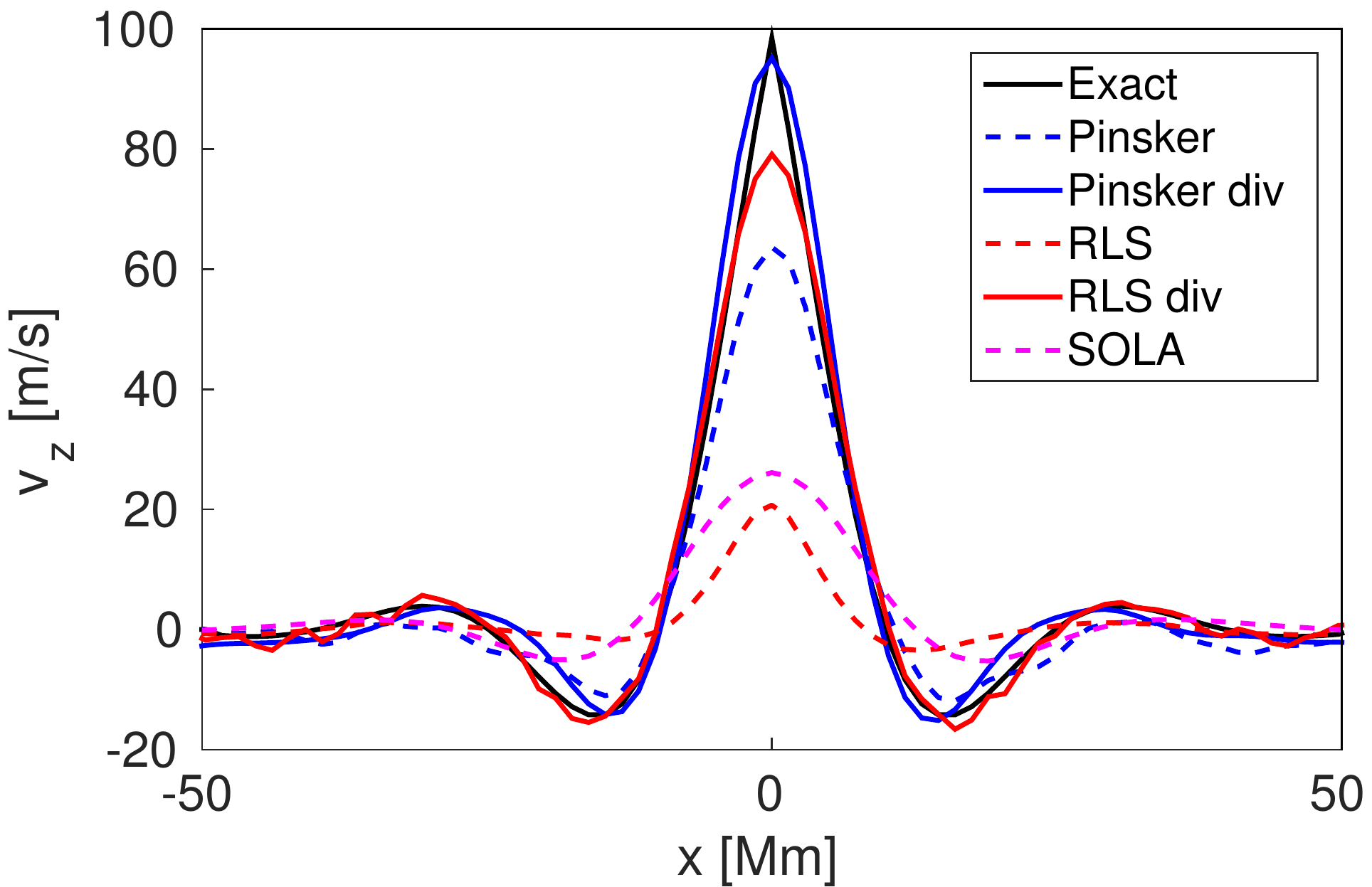} &  \includegraphics[width=0.45\textwidth]{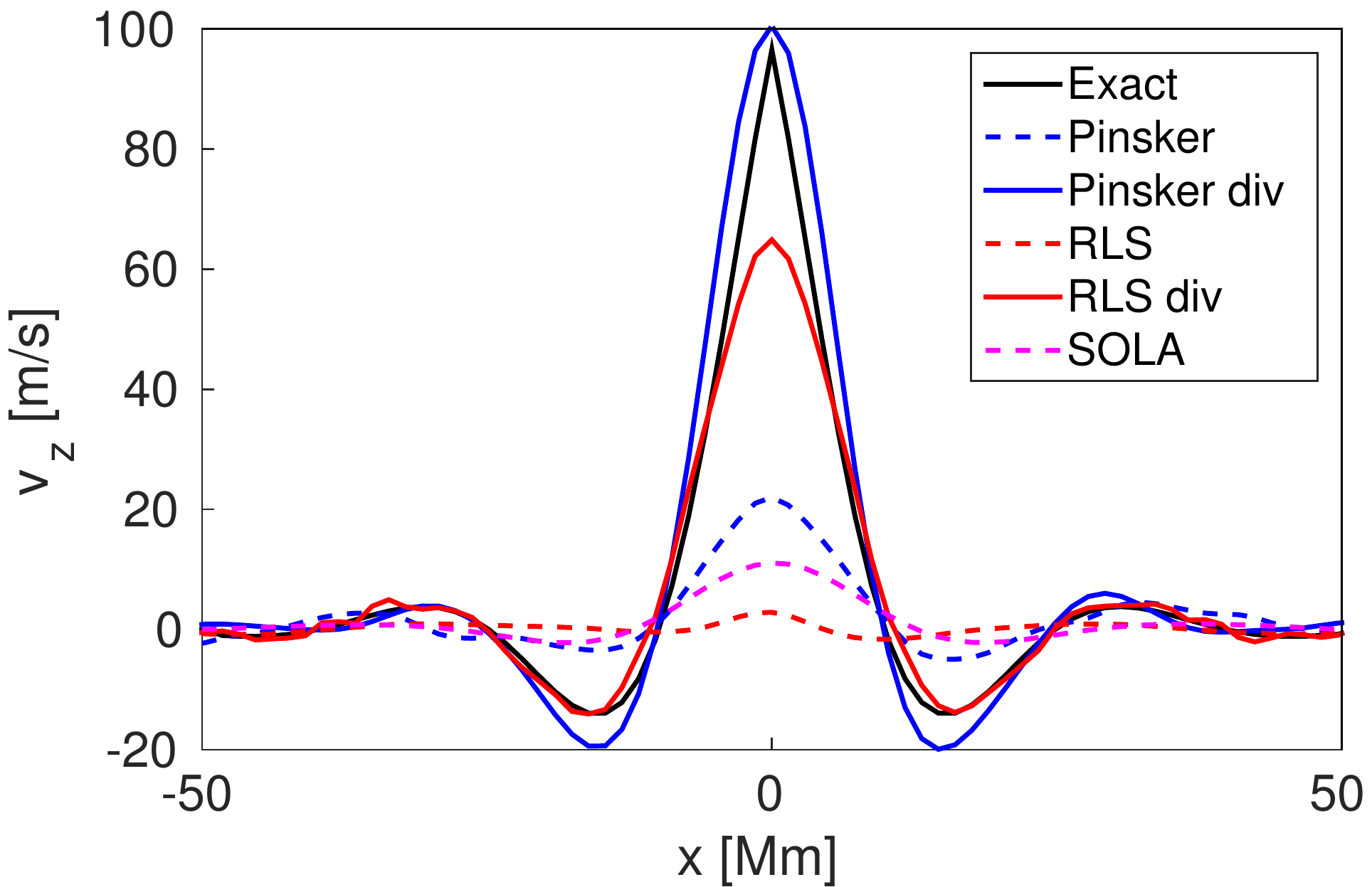}
  \end{tabular}
 \caption{Comparison of the different methods to reconstruct the vertical velocity $v_z(x,0,z_t)$ at $z_t = -3.5\Mm$ (left) and $z_t = -5.5\Mm$ (right). The Pinsker method with mass conservation provides the best reconstruction.}
\label{fig:vzCut}
 \end{figure}

Finally, we also study averaging kernels. Note that in the case of divergence 
constraints we have to think again about the definition of such kernels as 
$\delta$-peaks are not divergence free. 
We redefine the Fourier coefficients of the averaging kernel as
 \begin{equation}
  \Mc{K}_{k,\mathrm{div}} = (I-P_\kv) + P_\kv \Mc{K}_k P_\kv \label{eq:avgkernelDiv}
 \end{equation}
where $P_{\kv}$ denotes the $L^2$-orthogonal projection on the nullspace of $\divs_{\rho}$. 
This type of kernel still characterizes the bias of regularization methods if they 
are applied to solutions satisfying the mass conservation constraint and if 
$P_\kv$ is applied as a 
postprocessing step. Note that this definition implies that the Fourier coefficients 
of the averaging kernel are non-zero even at high frequencies due to the identity term. Thus, these averaging kernels cannot be directly compared to the ones of \Sec{helio}
(and thus to the ones classically used in helioseismology), but their definition using \Eq{avgkernelDiv} is natural as the convolution of $\Mc{K}_{\mathrm{div}}$ with the velocities characterizes the bias of the method. 

In \Fig{avgkernelDiv} we plot at each voxel the Frobenius norm of the 
$3\times 3$ matrix of averaging kernels of Pinsker with divergence constraints, 
i.e.\ the Euclidean norm of all the 9 averaging kernels at this voxel. 
Note that these kernels are very well localized even in $z$-direction and 
at $-5.5\Mm$. This explains the significant improvement in estimating vertical 
velocities achieved by the incorporation of the mass conservation constraint. 
We can even achieve a reasonable resolution in $z$-direction, which is not needed  
for reconstructing the supergranule model used as test example.  
\begin{figure}[ht]
\centering
\begin{tabular}{ccc}
$z_t = -0.9\Mm$ & $z_t = -3.5\Mm$ & $z_t = -5.5\Mm$ \\
 \includegraphics[trim=0cm 0cm 1cm 0cm, clip=true, width=0.3\linewidth]{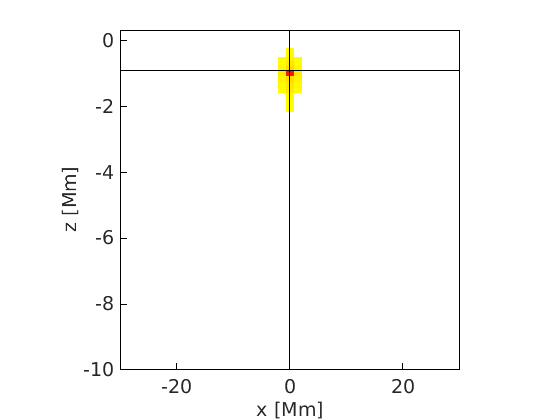} & \includegraphics[trim=0cm 0cm 1cm 0cm, clip=true, width=0.3\linewidth]{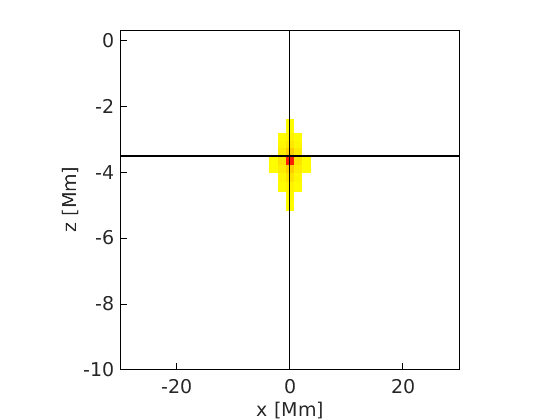} & \includegraphics[trim=0cm 0cm 1cm 0cm, clip=true, width=0.3\linewidth]{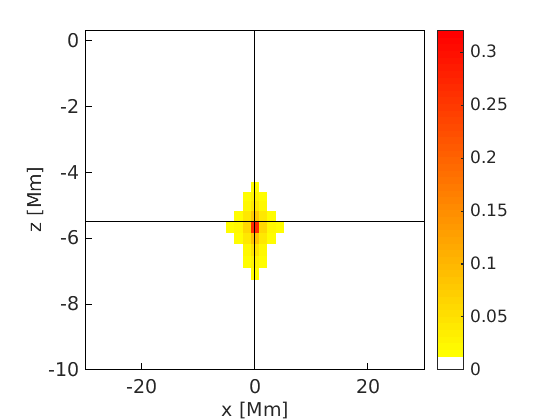}
\end{tabular}
 \caption{Pointwise Frobenius norm of the 3x3 averaging kernels $\Mc{K}_{\mathrm{div}}$ for the Pinsker method with mass conservation at three different depths $z_t \in\{-0.9\Mm, -3.5\Mm, -5.5\Mm\}$.}
\label{fig:avgkernelDiv}
\end{figure}

\section{Conclusions}
We have shown that Pinsker estimators yield significantly better reconstructions 
of vertical velocities from travel time maps than Tikhonov regularization, 
and is also superior to Subtractive Optimally Localized Averaging. 
This is consistent with theoretical 
optimality properties of these estimators. However,  as soon as depth inversion is involved, 
no simple, precise characterization of the ellipsoids on which Pinsker method is optimal 
is available. This is the usual situation for all spectral regularization methods 
such as Tikhonov regularization, Showalter's method, Landweber iteration, and many others 
in a deterministic context. As opposed to many other real-world problems, Pinsker 
estimators are computationally efficient and easy to implement in the context 
of local helioseismology. 

The mass conservation constraint can be incorporated naturally into the Pinsker 
estimator leading to another significant improvement of accuracy and resolution.
Under realistic noise levels this yields reliable estimators of vertical 
velocity components up to a depth of $-5.5\Mm$ using travel times from $f$ and $p_1$ to $p_4$ modes.

Alternatively, one may study an adaptive, data-driven choice of the size of the 
ellipsoid in the Pinsker method, which may be interpreted as a regularization parameter. 
We plan to address this as well as the application to real data in future work.  

\section*{Acknowledgement} The authors would like to thank Michal {\v S}vanda for providing the source codes and the sensitivity kernels used in the paper \cite{SVA11} and for helpful discussions. We would like to thank Takashi Sekii for useful comments on an earlier version of the manuscript.  
We acknowledge financial support from Deutsche Forschungsgemeinschaft through SFB-963 
``Astrophysical Flow Instabilities and Turbulence'' (Project A1). L.G. acknowledges support from the Center for Space Science, NYU Abu Dhabi
Institute, Abu Dhabi.

\section*{Appendix}
This appendix contains the proof of \Lem{helmholtz_cont}.

\begin{proof}
We make the substitutions $\bmom = \rho\bsol$ and $\bq = \rho\bw$. 
To prove \eref{it:grad_curl_div}, note that 
\[
\sum_{\beta\in\{\rm x,y,z\}} \langle\grad \mom^{\beta},\grad \qmom^{\beta}\rangle_{L^2(V)^3}
= \sum_{\beta,\gamma\in\{{\rm x,y,z}\}} \diffq{\mom^{\beta}}{\gamma} 
\overline{\diffq{\qmom^{\beta}}{\gamma}}
\]
and 
\begin{eqnarray*}
&&\langle \curl \bmom,\curl \bq\rangle_{L^2(V)^3}  + \langle \div \bmom,\div\bq\rangle_{L^2(V)}\\
&=& \sum_{\beta,\gamma\in\{{\rm x,y,z}\}} 
\int_V \diffq{p^{\beta}}{\gamma} \overline{\diffq{q^{\beta}}{\gamma}}\,\mathrm{d}x
+ \sum_{\beta\neq \gamma}\int_V\left[
\diffq{\mom^{\beta}}{\beta} \overline{\diffq{q^{\gamma}}{\gamma}}
-\diffq{\mom^{\beta}}{\gamma} \overline{\diffq{q^{\gamma}}{\beta}}
\right]\,\mathrm{d}x
\end{eqnarray*}
For all terms in the second sum (coming among other terms from the $\curl$ part) 
we can perform partial integrations without boundary terms to see 
that these terms vanish. 
(Note that this would not work without the Dirichlet boundary conditions for the $z$-components, 
e.g.\ for $\beta=\mathrm{x}$ and $\gamma=\mathrm{z}$.) 
Therefore, the left hand sides of the last two equations are equal. 

Part (\ref{it:norm_equivalence}): As 
\[
\|\bmom\|_{H^1(V)^3}^2 = 
 \sum_{\beta\in\{{\rm x,y,z}\}} \|\mom^{\beta}\|_{L^2(V)}^2+ \|\grad \mom^{\beta}\|_{L^2(V)^3}^2
\]
the second inequality in \eref{eq:norm_equivalence} 
follows from \eref{eq:grad_curl_div} and the Cauchy-Schwarz inequality 
$\int_V \mom^{\beta}\mathrm{d}(\mathbf{r},z)\leq \|\mom^{\beta}\|_{L^2} |V|^{1/2}$ 
for $\beta\in\{\mathrm{x,y}\}$. \\
To prove the first inequality in \eref{eq:norm_equivalence} it suffices to show that there exists 
a constant $C\geq 0$ such that 
\[
\|\mom^{\beta}\|_{L^2}\leq C \|\grad \mom^{\beta}\|_{L^2}
+ \frac{C(1-\delta_{\beta,z})}{|V|}\left|\int_V \mom^{\beta}\mathrm{d}(\mathbf{r},z)\right|
\qquad \mbox{for all }\beta\in \{{\rm x,y,z}\}
\]
and all $\bmom\in \mathbb{X}$. 
For $\beta={\rm z}$ this follows from the Poincar\'e inequality due to the 
Dirichlet boundary conditions, and for $\beta\in\{{\rm x,y}\}$ it is a consequence of the 
Poincar\'e-Wirtinger inequality. 

Part (\ref{it:ortho_cont}): 
To show orthogonality w.r.t.\ the inner product $\lsp\rho\bsol,\rho\bw\rsp_{L^2(V)}$, 
let $\bsol\in \Mc{N}_{\diamond}(\curls_{\rho})$. 
Then by potential theorems (see e.g.\ \cite[Thm. 3.37]{monk:03}) 
we have $\rho\bsol = \grad f$ for some $f\in H^1(V)$. 
It follows by partial integration that 
$\lsp\rho\bsol,\rho\bw\rsp_{L^2(V)}= \int_V f \divs_{\rho}\bw\D x = 0$ for all 
$\bw\in \Mc{N}_{\diamond}(\divs_{\rho})$ where the boundary terms vanish due to the boundary 
conditions. Hence, $\bsol\perp \Mc{N}_{\diamond}(\divs_{\rho})$ w.r.t.\ the weighted 
$L^2$ inner product. 
Together with \eref{eq:grad_curl_div} we also obtain orthogonality w.r.t.\ 
the $\Xspace$ inner product. \\
Let $\bsol\in\Xspace_{\diamond}$ satisfy $\lsp \rho\bsol,\rho\bw\rsp_{L^2(V)}=0$ for all 
$\bw\in \mathcal{N}_{\diamond}(\divs_{\rho})$ and all $\bw \in\mathcal{N}_{\diamond}(\curls_{\rho})$.  
We aim to show that $\bsol=0$.  
Since $\divs_{\rho}(\rho^{-1} \curl \mathbf{g})= 0$ and
$\curls_{\rho}(\rho^{-1} \grad f)=0$ for all smooth $f$ and $\mathbf{g}$ vanishing at the boundaries, 
we may choose $\bw=\rho^{-1}\curl\bg$ or $\bw = \rho^{-1}\grad f$ and perform 
partial integrations to 
obtain $\lsp \curls_{\rho} \bsol,\bg\rsp_{L^2(V)} = 0$ and $\lsp\divs_{\rho}\bsol,f\rsp =0$. 
Therefore $\curls_{\rho} \bsol=0$ and $\divs_{\rho}\bsol=0$, and so $\bsol=0$ 
from part (\ref{it:norm_equivalence}). This shows that the sum of the 
nullspaces is dense in $L_{\diamond}^2(V)$. Since the nullspaces are closed 
and orthogonal in $\Xspace_{\diamond}$, their sum equals $\Xspace_{\diamond}$. 
\end{proof}

\section*{References}
\bibliographystyle{plain}
\bibliography{biblio}

\end{document}